\documentclass[a4paper,12pt, reqno]{amsart} 
\usepackage{amssymb,amsthm,amsmath}
\usepackage{multirow}
\usepackage{enumerate}
\usepackage{ifthen}
\usepackage{graphicx} 
\usepackage{tkz-euclide}
\usepackage[colorlinks,citecolor=blue,hypertexnames=false]{hyperref}

 \numberwithin{equation}{section}

\usepackage{amsfonts}
\usepackage{amscd}
\usepackage{amssymb}
\usepackage{enumerate}
\usepackage{graphicx}
\allowdisplaybreaks
\usepackage{mathabx}
\usepackage{color}
\usepackage{amsbsy}
\usepackage{graphicx}
\usepackage{amsthm}
\usepackage{amsmath}
\usepackage{amsxtra}
\usepackage{mathrsfs}
\usepackage{bbm}
\usepackage{dsfont}

\usepackage{ifthen}
\usepackage{xcolor,colortbl}

\newtheorem{theorem}{Theorem}[section]

\theoremstyle{definition}

\theoremstyle{remark}
\newtheorem{rem}[theorem]{Remark}
\numberwithin{equation}{section}
\definecolor{red}{rgb}{1.0, 0.0, 0.0}
\newcommand{\Bea}{\begin{eqnarray*}}
	\newcommand{\Eea}{\end{eqnarray*}}
\newcommand{\Be} {\begin{equation*}}
	\newcommand{\Ee} {\end{equation*}}
\newcommand{\be} {\begin{equation}}
	\newcommand{\ee} {\end{equation}}
\newcommand{\bea} {\begin{eqnarray}}
	\newcommand{\eea} {\end{eqnarray}}

\newcommand{\HH}{\mathbb{H}^n}

\newcommand{\G}{\mathbb G}

\newcommand{\g}{\mathfrak{g}}

{\qed\bigskip}

\newcounter{alphabet}

\ifx\undefined\bysame
\newcommand{\bysame}{\leavevmode\hbox to3em{\hrulefill}\,}
\fi
\usetikzlibrary{patterns}

\title[Higher order hypoelliptic damped wave equations]
{Higher order hypoelliptic damped wave equations on   graded lie  groups with  data from negative order Sobolev spaces}
\author{Aparajita Dasgupta} 
\address{Aparajita Dasgupta \endgraf Department of Mathematics
	\endgraf Indian Institute of Technology  Delhi
	\endgraf Delhi, 110016  India.} 
\email{adasgupta@maths.iitd.ac.in}
\author{Vishvesh Kumar} 

\address{Vishvesh Kumar  \endgraf Department of Mathematics: Analysis, Logic and Discrete Mathematics	\endgraf Ghent University \endgraf Krijgslaan 281, Building S8,	B 9000 Ghent, Belgium.} \email{Vishvesh.Kumar@UGent.be and vishveshmishra@gmail.com}

\author{Shyam Swarup Mondal} \address{Shyam Swarup Mondal    \endgraf Department of Mathematics: Analysis, Logic and Discrete Mathematics\endgraf Ghent University\endgraf Krijgslaan 281, Building S8,	B 9000 Ghent, Belgium} \email{mondalshyam055@gmail.com}

\author{Michael Ruzhansky} \address{Michael Ruzhansky  \endgraf Department of Mathematics: Analysis, Logic and Discrete Mathematics\endgraf Ghent University\endgraf Krijgslaan 281, Building S8,	B 9000 Ghent, Belgium\endgraf and\endgraf School of Mathematical Sciences\endgraf Queen Mary University of London, United Kingdom} \email{michael.ruzhansky@ugent.be}

\date{\today}
\begin{document}
	\allowdisplaybreaks

	\begin{abstract} 
		Let $\mathbb G$ be a graded Lie group with homogeneous dimension $Q$. In this paper, we study the Cauchy problem for a semilinear hypoelliptic damped wave equation involving a positive Rockland operator $\mathcal{R}$ of homogeneous degree $\nu\geq 2$ on $\mathbb G$  with power type nonlinearity $|u|^p$ and initial data taken from negative order homogeneous Sobolev space $\dot H^{-\gamma}(\mathbb G), \gamma>0$.  In the framework of Sobolev spaces of negative order, we prove that $p_{\text{Crit}}(Q, \gamma, \nu) :=1+\frac{2\nu}{Q+2\gamma}$ is the new critical exponent for  $\gamma\in (0, \frac{Q}{2})$. More precisely, we show the global-in-time existence of small data Sobolev solutions of lower regularity for $p>p_{\text{Crit}}(Q, \gamma, \nu) $ in the energy evolution space $ \mathcal{C}\left([0, T], H^{s}(\mathbb{G})\right), s\in (0, 1]$. Under certain conditions on the initial data, we also prove a finite-time blow-up of weak solutions for $1<p<p_{\text{Crit}}(Q, \gamma, \nu)$.  Furthermore, to precisely characterize the blow-up time, we derive sharp upper bound and lower bound estimates for the lifespan in the subcritical cases.  We emphasize that our results are also new, even in the setting of higher-order differential operators on $\mathbb{R}^n$,  and more generally, on stratified Lie groups.
	\end{abstract}
	\maketitle
	\tableofcontents

 \section{Introduction}\label{sec1}
	In this work, we are interested in the study of the Cauchy problem for the semilinear damped wave equation with power-type nonlinearities:
	\begin{align} \label{eq0010}
		\begin{cases}
			u_{tt}+\mathcal{R}u +u_{t} =|u|^p, & x\in  \mathbb{G},~t>0,\\
			u(0,x)=\varepsilon u_0(x),  & x\in  \mathbb{G},\\ u_t(0, x)=\varepsilon u_1(x), & x\in  \mathbb{G},
		\end{cases}
	\end{align}
	where  $1<p<\infty$, $\mathcal{R}$ is a positive Rockland operator of homogeneous degree $\nu \geq 2$ on a graded Lie group $\mathbb{G}$,  and the initial data $(u_0, u_1)$ with its size parameter $\varepsilon>0$  belongs to  homogeneous Sobolev spaces of negative order $  \dot {H}^{-\gamma}(\mathbb{G}) \times  \dot {H}^{-\gamma}(\mathbb{G})$ with $\gamma>0$. Here we note that the commutative group $(\mathbb{R}^n,+),$ the Heisenberg group,  Engel groups, Cartan groups, and more generally, stratified Lie groups are examples of graded Lie groups.

	M\"uller and Stein \cite{MS99} initiated the analysis of wave equations associated with sub-Laplacian on the Heisenberg group  and proved $L^p$-estimates of wave kernels. Later, M\"uller and Seeger \cite{MS15} proved sharp $L^p$-estimates of wave kernels on groups of Heisenberg type. 
	Bahouri, Ge\'rard, and Xu \cite{Bahouri} began a systematic study of damped wave equations involving sub-Laplacian on the Heisenberg group  to prove a weak type decay rate in the dispersive estimate and local Strichartz estimates, which are mainly based on the analysis in Besov type spaces. In \cite{Fermanian}, the authors extended these results in the setting of step 2 stratified Lie groups and showed that the decay rate of the solution might be influenced by the dimension of the center of the group. Nevertheless, a better decay rate can be achieved when the wave equation for the full Laplacian on the Heisenberg group is taken into account, see \cite{FMV, Manli}.  See also \cite{manli1} for decay estimates for a class of wave equations on the Heisenberg group. 
 We also refer to \cite{attract1, attract2}, where damped hyperbolic equations were studied using attractors.

	In \cite{Vla}, Georgiev and Palmieri investigated the global existence and nonexistence results for the Cauchy problem for the semilinear damped wave equation on the Heisenberg group with the power type nonlinear (such as  $|u|^p$) by finding the critical exponent of Fujita-type. The meaning of critical exponent is the threshold condition on the exponent $p$ for the global-in-time Sobolev solutions and the blow-up of local-in-time weak solutions with small data.  Its counterpart in the Euclidean framework, the critical exponent for the semilinear damped wave equation with the power type nonlinear, can be found in \cite{IKeta and Tanizawa, Matsumura, Zhang, Todorova} and references therein. Moreover, Palmieri  \cite{Palmieri 2020} derived the $L^2$-decay estimates for solutions as well as for their time derivative, and their horizontal gradient of solution to the homogeneous linear damped wave equation on the Heisenberg group.  In \cite{30}, the authors considered the Cauchy problem for the semilinear damped wave equation for the sub-Laplacian on the Heisenberg group with power-like nonlinearities and proved global in time well-posedness for small data in the case of the positive mass and damping term.  Moreover,   the estimates for the linear viscoelastic type damped wave equation on the Heisenberg group can be found in  \cite{32}.  We also refer to \cite{27, 28,garetto, RY20, BKM22, DKM23} concerning the damped wave equation on compact Lie groups, \cite{gra1, gra3} for wave equations on graded Lie groups,  and \cite{AKR22} for Riemannian symmetric spaces of non-compact type.


	In recent years, considerable attention has been devoted by several researchers to find new critical exponents for the semilinear damped wave equations in different frameworks. For instance, in the Euclidean setting, for the semilinear damped wave equations with initial data belonging to $L^m$-regular data instead of $L^1$-data with $m \in(1,2)$, the critical exponent is changed from $p_{\text {Crit}}(n)=1+\frac{2 }{n}$ to the new modified critical exponent $p_{\text {Crit}}(n):=p_{\text {Crit }}\left(\frac{n}{m}\right)=1+\frac{2 m}{n}$, see \cite{Ikeda2002,Ikeda2019,Nakao93}.  Furthermore, for initial data additionally belonging to homogeneous Sobolev spaces $\dot H^{-\gamma}(\mathbb{R}^n)$ of negative order  $-\gamma$, the new critical exponent becomes $ p_{\text{Crit}}(n, \gamma):= 1+ \frac{4}{n+2\gamma}$ for some  $\gamma\in (0, \frac{n}{2})$, see \cite{Reissig} for more detail. In the non-Euclidean settings, specifically, for the semilinear damped wave equation associated with sub-Laplacian on the Heisenberg group $\mathbb{H}^n$, we \cite{DKMR} proved that the critical exponent $1+ \frac{2}{Q}$ for $L^1$-regular data is changed into the modified critical exponent $ p_{\text{Crit}}(Q, \gamma):= 1+ \frac{4}{Q+2\gamma}$ for some  $\gamma\in (0, \frac{Q}{2})$ if initial data is localized in the homogeneous Sobolev spaces of negative order   $\dot H^{-\gamma}(\mathbb{H}^n)$, where $Q:=2n+2$ is the homogeneous dimension on $\mathbb{H}^n.$


One of the important features in the negative order Sobolev space $\dot{H}^{-\gamma}, \gamma>0,$ is that the corresponding Sobolev norm of initial data enhances the faster decay rate of the solution compared to the $L^2$-decay. 
The study of different kinds of partial differential equations in the framework of Sobolev space of negative order  is not a new topic.  For example, one can see \cite{GW, Umakoshi, Yao} for semilinear heat equation with initial data in negative Sobolev spaces. We also refer to \cite{GW, Tao,Ohhh,Wang} and references therein for recent works in the direction of negative order Sobolev space.

	The main aim of this paper is to find  a critical exponent for the semilinear hypoelliptic damped wave equation (\ref{eq0010})  associated with a positive Rockland operator $\mathcal{R}$  on a graded Lie group $\G$   with power type nonlinearity $|u|^p$ and initial data taken from negative order homogeneous Sobolev space $\dot H^{-\gamma}(\mathbb G), \gamma>0$. 
 
 Recall that a connected and simply connected Lie group $\mathbb{G}$ is a graded Lie group if its Lie algebra $\mathfrak{g}$ is {\it graded}, that is, $\mathfrak{g}$ admits a vector space decomposition of the form $\mathfrak{g}= \bigoplus_{i=1}^\infty \mathfrak{g}_i,$ for which all but finitely many $\mathfrak{g}_i$'s are $\{0\}$ such that $[\mathfrak{g}_i, \mathfrak{g}_j] \subset \mathfrak{g}_{i+j}$ for all $i, j \in \mathbb{N}.$ We refer to Section \ref{sec2} for a detailed description of the graded Lie groups.  If the first stratum $\mathfrak{g}_1$  generates the Lie algebra $\g$ as an algebra, the group $\mathbb{G} $ is called a {stratified Lie group}. In this case, the sum of squares of a basis of vector fields in $\mathfrak{g}_1$ gives a sub-Laplacian on $\mathbb{G}$. This immediately shows that every stratified Lie group is graded. However, if the group $\mathbb{G}$ is non-stratified, then it may not have a homogeneous sub-Laplacian or Laplacian but they always posses Rockland operators. 	A Rockland operator on $\mathbb{G}$ is a left-invariant hypoelliptic differential operator of a positive homogeneous degree $\nu$, see Subsection \ref{Rocksec} for an overview.  
	The Heisenberg group, more generally, $H$-type groups, Engel groups, and Cartan groups are examples of graded Lie groups. The following are some examples of graded Lie groups with a Rockland operator  which are included in the analysis of this paper.
 
	\begin{itemize}
		\item When $\mathbb{G}=(\mathbb{R}^n,+),$ a Rockland operator $ \mathcal{R}$ can be any positive homogeneous elliptic differential operator with constant coefficients, for example, we can  consider
\begin{align}\label{R^n}
		\mathcal{R}=(-\Delta)^m \text { or } \mathcal{R}=(-1)^m \sum_{j=1}^n a_j\left(\frac{\partial}{\partial x_j}\right)^{2 m}, \quad a_j>0, m \in \mathbb{N},
\end{align}
  which are Rockland operators with homogeneous degree $2m$ when the commutative group $\mathbb{R}^n$ is equipped with isotropic dilations. 
		\item When  $\mathbb{G}=\mathbb{H}^n,$ the Heisenberg group, we can consider the Rockland operator of the homogeneous degree $2m$ as  
$$ 		\mathcal{R}=(-\mathcal{L})^m \text { or } \mathcal{R}=(-1)^m \sum_{j=1}^n\left(a_j X_j^{2 m}+b_j Y_j^{2 m}\right), \quad a_j, b_j>0, m \in \mathbb{N},
$$		where $X_j=\partial_{x_j}-\frac{y_j}{2} \partial_t,  Y_j=\partial_{y_j}+\frac{x_j}{2} \partial_t$ are the    left-invariant vector fields for its algebra $\g$ and     	$\mathcal{L}=\sum_{j=1}^n\left(X_j^2+Y_j^2\right)$  is the sub-Laplacian on    $\mathbb{H}^n$.\\
		\item When  $\mathbb{G}$ is a stratified Lie group,  then $\mathcal{L}_{\mathbb G}$, defined in  (\ref{stratified})  
  is  a positive Rockland operator with    homogeneous degree $\nu = 2$.\\
		\item When $\mathbb{G}$ is a graded Lie group with dilation weights $\nu_1, \ldots, \nu_n$,  if $\nu_0$ is any common multiple of $\nu_1, \ldots, \nu_n$, then the operators given by 
		\begin{equation}
			\mathcal{R}:= \sum_{j=1}^n(-1)^{\frac{v_0}{v_j}} a_j X_j^{2 \frac{v_0}{v_j}}, \quad \text{with}\,\, a_1, a_2, \ldots, a_n>0,
		\end{equation}
		are  positive Rockland operators of homogeneous degree $\nu=2\nu_0$  	for any strong Malcev basis $\{X_1, X_2, \ldots, X_n\}$ of the Lie algebra $\g$.
	\end{itemize}
	

	For the semilinear hypoelliptic damped wave equation (\ref{eq0010})  for a positive Rockland operator $\mathcal{R}$  of homogeneous degree $\nu$   with power type nonlinearity $|u|^p$ on the graded Lie group $\mathbb{G}$, we show that  $p_{\text{Crit}}(Q, \gamma, \nu):=1+\frac{2\nu}{Q+2\gamma}$ is the critical exponent for some  $\gamma\in (0, \frac{Q}{2})$  	when the initial data are additionally taken from homogeneous Sobolev space of negative order $\dot {H}^{-\gamma}(\mathbb{G} )$. We note that our analysis on this new critical exponent $p_{\text{Crit}}(Q, \gamma, \nu):=1+\frac{2\nu}{Q+2\gamma}$ generalizes several known results. For example,
\begin{itemize}
    \item {\bf  When $\mathbb G=\mathbb{R}^n$ and $\mathcal{R}=-\Delta$}: the critical exponent   $p_{\text{Crit}}(Q, \gamma, \nu)$ coincides with the critical exponent 
 $p_{\text{Crit}}(n, \gamma):= 1+ \frac{4}{n+2\gamma}$ for some  $\gamma\in (0, \frac{n}{2})$ considered in  \cite{Reissig}. 
    \item {\bf When  $\mathbb{G}=\mathbb{H}^n,$ the Heisenberg group and $\mathcal{R}=-\mathcal{L}$, sub-Laplacian on  $\mathbb{H}^n$}:   the critical exponent  $p_{\text{Crit}}(Q, \gamma, \nu)$ reduce to the critical exponent $ p_{\text{Crit}}(Q, \gamma):= 1+ \frac{4}{Q+2\gamma}$ for some  $\gamma\in (0, \frac{Q}{2})$ considered in \cite{DKMR}.
  \item The critical exponent $p_{\text{Crit}}(Q, \gamma, \nu)$  is new even in the setting of higher-order homogeneous differential operators (such as powers of negative Laplacian) (\ref{R^n})
on $\mathbb{R}^n$, and, more generally, for a negative sublaplacian and its powers
on a stratified Lie group  $\mathbb{G}$.
\end{itemize}  
The following $\dot {H}^s(\G)$-norm estimate of the solution will play a vital role in the proof of the global-in-time existence result.
	\begin{theorem}\label{Linear} Let ${\mathbb{G}}$ be a graded Lie group of homogeneous dimension $Q$ and let $\mathcal{R}$ be a positive Rockland operator of homogeneous degree $\nu.$
		Assume that $\left(u_0, u_1\right) \in\ (H ^s \cap \dot {H} ^{-\gamma}\ ) \times\ (H ^{s-1} \cap  \dot{H} ^{-\gamma}\ )$ with  $s \geq 0$ and $s+\gamma \geq  0$. 
		Then the  solution  of  the   linear Cauchy problem   	\begin{align}\label{Linear-system}
			\begin{cases}
				u_{tt}+\mathcal{R}u +u_{t} =0, & x\in  \mathbb{G},~t>0,\\
				u(0,x)= u_0(x),  & x\in  \mathbb{G},\\ u_t(0, x)= u_1(x), & x\in  \mathbb{G},
			\end{cases}
		\end{align}
		satisfies the following 
		$ \dot{H} ^s$-decay estimate	\begin{align}\label{hom}
			\|u(t, \cdot)\|_{ \dot {H}^s} \lesssim(1+t)^{-\frac{s+\gamma}{\nu}}\left(\left\|u_0\right\|_{H^s \cap\dot {H}^{-\gamma}}+\left\|u_1\right\|_{H^{s-1} \cap \dot{H}^{-\gamma}}\right) ,
		\end{align}
		for any $t\geq 0$. 
	\end{theorem}
	Using the above $\dot {H}^s$-norm estimate, we will prove a global-in-time existence of small data Sobolev solutions to (\ref{eq0010})  of lower regularity as follows:
	\begin{theorem}\label{well-posed}  Let ${\mathbb{G}}$ be a graded Lie group of homogeneous dimension $Q$ and let $\mathcal{R}$ be a positive Rockland operator of homogeneous degree $\nu \geq 2.$
		Let $s \in(0,1]$ and $\gamma \in\left(0, \frac{Q}{2}\right)$. Assume that an exponent $p$ satisfies
		\begin{align}\label{eq24}
			1<p \leq \frac{Q}{Q-2 s}\quad\text{and} \quad 	p\left\{\begin{array}{ll}
				>p_{\text {Crit }}(Q, \gamma, \nu):=1+\frac{2\nu}{Q+2\gamma} & \text { if } \gamma \leq \tilde{\gamma}, \\
				\geq 1+\frac{2 \gamma}{Q} & \text { if } \gamma>\tilde{\gamma}, 
			\end{array}\right. 
		\end{align}
		where $\tilde{\gamma}$ denotes the positive root of  the quadratic equation $2 \tilde{\gamma}^2+Q \tilde{\gamma}-\nu Q=0$, i.e., $\tilde{\gamma}= \frac{-Q+\sqrt {Q^2+8\nu Q}}{4}.$

  Then, there exists a small positive constant $\varepsilon_0$ such that for any $\left(u_0, u_1\right) \in \mathcal{A}^{s }:=\ (H^s \cap \dot{H}^{-\gamma}\ ) \times\ (L^2 \cap \dot{H}^{-\gamma}\ )$ satisfying $\left\|\left(u_0, u_1\right)\right\|_{\mathcal{A}^{s}}=\varepsilon \in\left(0, \varepsilon_0\right]$, the Cauchy problem for the semilinear damped wave equation (\ref{eq0010}) has a uniquely determined Sobolev solution
	$$
	u \in \mathcal{C}\left([0, \infty), H^s\right).
	$$

	\end{theorem} 	
	Note that the technical restriction on  $1< p\leq \frac{Q}{Q-2s}$  in the above theorem is due to an application of the  Gagliardo-Nirenberg type inequality. Moreover, some examples for the admissible range of the exponent $p$ for the global-in-time existence result in certain low homogeneous dimension graded Lie group $\mathbb{G}$ are as follows:
	\begin{itemize} 
		\item When $Q=1, 2$, we take $s \in(0,1]$ and $\gamma \in (0,\frac{Q}{2} )$ and the exponent satisfies
		
		$$
		1+\frac{2\nu}{Q+2 \gamma}<p   \left\{\begin{array}{ll}
			<\infty & \text { if } Q \leq 2s, \\
			\leq   \frac{Q}{Q-2s}& \text { if } Q>2s.
		\end{array}\right.  
		$$
		\item When $Q=3, 4$, we take $s \in(0,1]$ and $\gamma \in (0, \frac{Q}{2})$ and the exponent satisfies
		
		$$
		\begin{array}{l} \vspace{0.3cm}
			1+\frac{2\nu}{Q+ 2\gamma}<p \leq  \frac{Q}{Q- 2s} ~\quad \text { if } ~0<\gamma \leq \tilde{\gamma}  \\ 
			1+\frac{ 2\gamma}{Q} \leq  p \leq \frac{Q}{Q- 2s} ~\quad\quad \text { if }  ~\tilde{\gamma} <\gamma<\frac{Q}{2} .
		\end{array}
		$$
		
		%
\end{itemize}


 Let  $|\cdot|$   be any homogeneous norm on the graded Lie group  $\mathbb{G}$, while we denote $\left(1+|x|^2\right)^{\frac{1}{2}}$ by the Japanese bracket      $\langle x\rangle$ for $x\in \mathbb{G}$. Our next result is about the blow-up (in-time)  of weak solutions to the Cauchy problem  (\ref{eq0010})  in the subcritical case $1<p<p(Q, \gamma, \nu)$ under certain additional assumptions on the initial data. 
\begin{theorem} \label{blow-up} Let ${\mathbb{G}}$ be a graded Lie group of    homogeneous dimension $Q$ and let $\mathcal{R}$ be a positive Rockland operator given by
\begin{equation}
			\mathcal{R}:= \sum_{j=1}^n(-1)^{\frac{\nu_0}{\nu_j}} a_j X_j^{2 \frac{\nu_0}{\nu_j}}, \quad \text{with}\,\, a_1, a_2, \ldots, a_n>0,
		\end{equation}
of homogeneous degree $\nu:=2\nu_0,$ where  $\nu_0$ is any common multiple of dilations weights $\nu_1, \ldots, \nu_n$ on $\G$ and  $\{X_1, X_2, \ldots, X_n\}$ is a strong Malcev basis of the Lie algebra $\g$ of $\G.$  Let $\gamma \in\left(0, \frac{Q}{2}\right)$ and the exponent $p$ satisfies $1<p<p(Q, \gamma, \nu) :=1+\frac{2\nu}{Q+2\gamma}$. We   also assume  that the non-negative initial data $\left(u_0, u_1\right) \in \dot{H} ^{-\gamma} \times \dot{H} ^{-\gamma}$ satisfies 
	\begin{align}\label{eq32}
		u_0(x)+u_1(x) &\geq C_1 \langle x \rangle^{-Q\left(\frac{1}{2}+\frac{\gamma}{Q}\right)}(\log (e+|x|))^{-1}, \quad x\in  \mathbb G,
	\end{align}
	where $C_1$ is a positive constant. Then, there is no global (in-time) weak solution to (\ref{eq0010}). 
\end{theorem}

\begin{table}[h] \label{tabe}
\setlength{\tabcolsep}{12pt}
\renewcommand{\arraystretch}{2}
\begin{tabular}{|c|c|c|c|c|}
\hline
\cellcolor{gray!25} $Q$ 
& \cellcolor{gray!25} $\nu$ 
& \cellcolor{magenta!70} Global Existence 
&\cellcolor{red!65}  Blow-up\\
\hline
1, 2   &  $\geq 2$ &     $ 
1+\frac{2\nu}{Q+2 \gamma}<p    
\leq   \frac{Q}{(Q-2s)_+} 
$    &  $1<p<1+\frac{2\nu}{Q+2 \gamma}$ \\ 
\hline
3 &  2 &  $ 
\begin{array}{l}
	1+\frac{4}{3+ 2\gamma}<p \leq  \frac{Q}{Q- 2s} ~\quad \text { if } ~0<\gamma \leq \tilde{\gamma}  \\ 
	1+\frac{ 2\gamma}{Q} \leq  p \leq \frac{Q}{Q- 2s} ~\quad\quad \text { if }  ~\tilde{\gamma} <\gamma<\frac{Q}{2} .
\end{array}
$  &  $  1<p<1+\frac{4}{3+2 \gamma}$   \\ 
\hline
3    &  4&  $ 
1+\frac{8}{3+2 \gamma}<p    
\leq   \frac{Q}{(Q-2s)} 
$    &    $1<p<1+\frac{8}{3+2 \gamma}$ \\
\hline
4, 5, 6&   2  &  $ 
\begin{array}{l} \vspace{0.3cm}
	1+\frac{2\nu}{Q+ 2\gamma}<p \leq  \frac{Q}{Q- 2s} ~\quad \text { if } ~0<\gamma \leq \tilde{\gamma}  \\ 
	1+\frac{ 2\gamma}{Q} \leq  p \leq \frac{Q}{Q- 2s} ~\quad\quad \text { if }  ~\tilde{\gamma} <\gamma<\frac{Q}{2} .
\end{array}
$  &  $1<p<1+\frac{2\nu}{Q+2 \gamma}$ \\
\hline
\end{tabular}
\vspace{10pt}
\caption{Ranges of  $p$ for global-in-time
existence and blow-up of weak solutions for a pair $(Q, \nu).$}
\vspace{-15pt}
\label{Table3}
\end{table}

Let $s \in(0,1]$ and $\gamma \in\left(0, \frac{Q}{2}\right)$. To provide a rough idea about the critical exponent for a pair of $(Q, \nu)$, from Theorem \ref{well-posed} and Theorem \ref{blow-up}, the ranges of the exponent $p$ for global-in-time existence and blow-up of weak solution to the Cauchy problem (\ref{eq0010}) are discussed in Table \ref{tabe}.

In particular, when $s=1$ and $\nu=2$,  for a complete analysis of the critical exponent which depends on the parameter $\gamma$ and $Q$, we can describe it by the $(\gamma, p) $ plane in Figure \ref{imgg2}.	In  Figure \ref{imgg2},  with an  increase of the homogeneous dimension $Q$ of $\mathbb{G}$, the curve $p=p_{\text {Crit }}(Q, \gamma, 2)$ and the segment $p=1+\frac{2 \gamma}{Q}$ will move following the direction of   $\nearrow$ and $\nwarrow$ lines arrows, respectively.

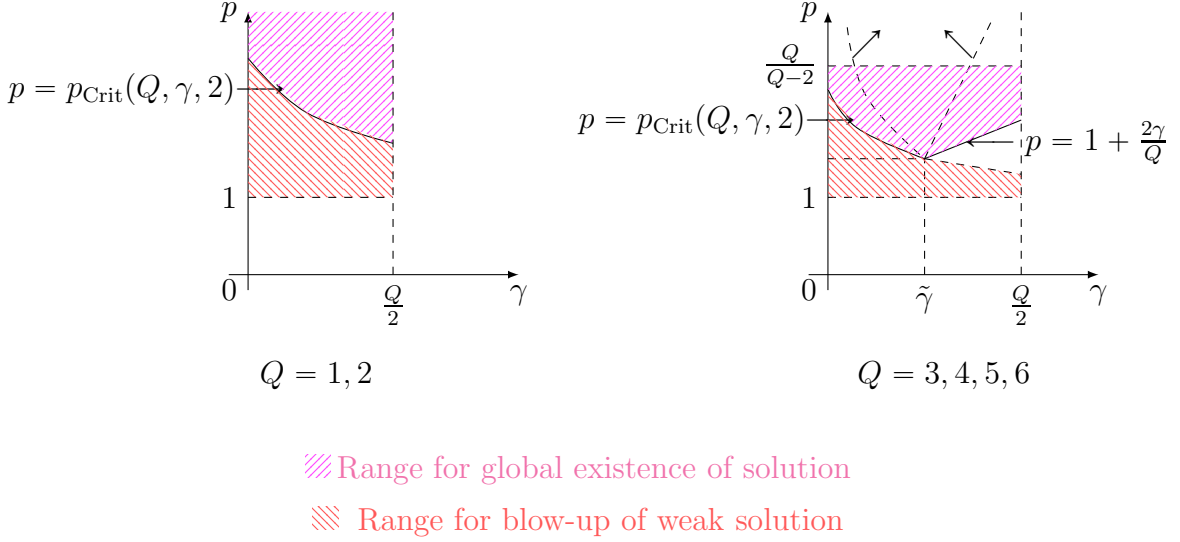
\begin{figure}[t]
\centering
\begin{tikzpicture}[>=latex,xscale=1.55,yscale=1.25,scale=0.82]
	\draw[->] (-0.2,0) -- (2.8,0) node[below] {$\gamma$};
	\draw[->] (0,-0.2) -- (0,3.4) node[left] {$p$};
	\node[right, color=black] at (-2.7,2.4) {{ $p=p_{\mathrm{Crit}}(Q,\gamma, 2)$$\longrightarrow$}};
	\node[left] at (0,-0.2) {{$0$}};
	\draw[color=black] plot[smooth, tension=.7] coordinates {(0,2.8) (0.6,2.1) (1.5,1.7)};
	\fill [pattern=north west lines, pattern color=red!65]  (0,2.8) -- (0.6, 2.1) -- (1.5,1.7) -- (1.5, 1) -- (0,1);
	\fill [pattern=north east lines, pattern color=magenta!70]  (0,2.8) -- (0.6, 2.1) -- (1.5,1.7) -- (1.5,3.4) -- (0, 3.4);
	\node[below] at (1.5,0) {{$\frac{Q}{2}$}};
	\node[left] at (0,1) {{$1$}};
	\draw[dashed, color=black]  (0, 1)--(1.5, 1);
	\draw[dashed, color=black]  (1.5, 0)--(1.5, 3.4);
	\draw[->] (5.8,0) -- (8.8,0) node[below] {$\gamma$};
	\draw[->] (6,-0.2) -- (6,3.4) node[left] {$p$};
	\node[left] at (6,-0.2) {{$0$}};
	\node[left] at (6,2.7) {{$\frac{Q}{Q-2}$}};
	\node[below] at (7,0) {{${\tilde{\gamma}}$}};
	\node[below] at (8,0) {{$\frac{Q}{2}$}};
	\node[left] at (6,1) {{$1$}};
	\node[left, color=black] at (9.64,1.7) {{$\longleftarrow$ $p=1+\frac{2\gamma}{Q}$}};
	\node[right, color=black] at (3.2,2) {{ $p=p_{\mathrm{Crit}}(Q,\gamma,2)$$\longrightarrow$}};
	\draw[dashed, color=black]  (6, 1)--(8, 1);
	\draw[dashed, color=black]  (8, 0)--(8, 3.4);
	\draw[dashed, color=black]  (6, 2.7)--(8, 2.7);
	\draw[dashed, color=black]  (7, 0)--(7, 1.5);
	\draw[dashed, color=black] (6,1.5)--(7,1.5);
	\draw[color=black] plot[smooth, tension=.7] coordinates {(6,2.4) (6.3,1.9) (7,1.5)};
	\draw[dashed, color=black] plot[smooth, tension=.7] coordinates {(6.2,3.2) (6.4,2.3) (7,1.5)};
	\node[below] at (6.4,3.3) {{$\nearrow$}};
	\node[below] at (7.35,3.3) {{$\nwarrow$}};
	\draw[color=black] plot[smooth, tension=.7] coordinates { (7,1.5) (8,2)};
	\draw[dashed, color=black] plot[smooth, tension=.7] coordinates { (7,1.5) (7.7,3.3)};
	\draw[dashed, color=black] plot[smooth, tension=.7] coordinates { (7,1.5) (7.5,1.4) (8,1.3)};
	\fill [pattern=north west lines, pattern color=red!65]  (6,2.4) -- (6.3,1.9) -- (7,1.5) -- (7.5,1.4) --(8,1.3) -- (8,1) -- (6, 1);
	\fill [pattern=north east lines, pattern color=magenta!70]  (6,2.4) -- (6.3,1.9) -- (7,1.5) -- (8,2) -- (8,2.7) -- (6,2.7);
	\node[left] at (1.4,-1.3) {{  $Q=1,2$}};
	\node[left] at (8.2,-1.3) {{  $Q=3,4,5,6$}};
\end{tikzpicture}
\begin{center}

	\begin{tikzpicture}
	\end{tikzpicture}\\
	\begin{tikzpicture}[>=latex,xscale=1.35,yscale=1.25,scale=0.82]	\node[left] at (4.81,4.1) {{\color{magenta!70}  Range for global existence of solution}};\fill [pattern=north east lines, pattern color=magenta!70]  (-1.51,4) -- (-1.51, 4.3) -- (-1.81,4.3) -- (-1.81, 4);	\end{tikzpicture}\\
	\begin{tikzpicture}[>=latex,xscale=1.35,yscale=1.25,scale=0.82]	\node[left] at (4.8,4.1) {{\color{red!65}  Range for blow-up of weak solution}}; \fill [pattern=north west lines, pattern color=red!65]   (-1.34,4) -- (-1.34, 4.3) -- (-1.64,4.3) -- (-1.64, 4); \end{tikzpicture}
\end{center}
\caption{Description of the critical exponent in the $(\gamma, p)$ plane for $\nu=2$}
\label{imgg2}\end{figure}



From  Theorems \ref{well-posed} and Theorem \ref{blow-up}, we can conclude that the critical exponent for the semilinear damped wave equation (\ref{eq0010})   is 	$p_{\text{Crit}}(Q, \gamma, \nu)=1+\frac{2\nu}{Q+2\gamma}$  when initial data are additionally taken from the negative order Sobolev space  $ \dot{H}^{-\gamma}, \gamma>0$. It gives us a new way to look at the critical exponent in Sobolev space with a negative order for the semilinear damped wave equation on graded Lie groups.   

Finally, we derive sharp lifespan estimates for weak solutions to the semilinear Cauchy problem (\ref{eq0010}). We define  the lifespan  $T_\varepsilon$ as the maximal existence time for solution of (\ref{eq0010}), i.e.,
\begin{align}\nonumber \label{eq42}
T_\varepsilon :=\sup \Big\{ T>0~&:~ \text{there exists a unique local-in-time solution to the Cauchy} \\& \text{  problem (\ref{eq0010}) on $[0, T)$   with a fixed parameter}~\varepsilon>0\Big\}.
\end{align}
We denote $T_{w,\varepsilon}$ and $T_{m,\varepsilon}$ as the lifespan for a weak and mild solution to the Cauchy problem (\ref{eq0010}), respectively. Then, for $\gamma \in (0, \tilde{\gamma})$ and $1<p<p_{\text {Crit}}(Q, \gamma, \nu)$, we prove   the following lower bound and upper bound estimates for the lifespan of the weak solution and the mild solution.

\begin{itemize}
\item The non-trivial local-in-time weak solution blows up in finite time, and the  lifespan $T_{w, \varepsilon}$ of a weak solution to   (\ref{eq0010}) satisfies the following upper bound estimate:
$$ 	T_{w,\varepsilon}  \lesssim  \varepsilon^{-\left(\frac{1}{p-1}-\left(\frac{Q}{2\nu}+\frac{\gamma}{\nu}\right)\right)^{-1}} .$$

\item Additionally, if the exponent $p$ satisfies
$
1+\frac{2 \gamma}{Q} \leq p \leq \frac{Q}{Q-2 }
$ and  $\left(u_0, u_1\right) \in \mathcal{A}^1$,   then the  lifespan $T_{m, \varepsilon}$ of a mild solution   to   (\ref{eq0010}) satisfies the following lower bound estimate:
$$
T_{m,\varepsilon}  \gtrsim  \varepsilon^{-\left(\frac{1}{p-1}-\left(\frac{Q}{2\nu}+\frac{\gamma}{\nu}\right)\right)^{-1}} .
$$

\end{itemize}
Since a solution $u$ in Theorem \ref{well-posed} is a mild solution to \eqref{eq0010},  this mild solution is also a weak solution to (\ref{eq0010}) with $ T_{m,\varepsilon}\leq  T_\varepsilon$.  Consequently, if  the initial data    is  from  negative order homogeneous Sobolev space $\dot  {H} ^{-\gamma}(\mathbb{G})$  and the exponent   $p$ satisfies  $1+\frac{2 \gamma}{Q} \leq p \leq \frac{Q}{Q-2}$,  then  we can claim the sharp estimate for the lifespan $T_\varepsilon$ as  
$$
T_\varepsilon\left\{\begin{array}{ll}
=\infty & \text { if } p>p_{\text {Crit }}(Q, \gamma, \nu), \\
\simeq   C \varepsilon^{-\left(\frac{1}{p-1}-\left(\frac{Q}{2\nu}+\frac{\gamma}{\nu}\right)\right)^{-1}}  & \text { if } p<p_{\text {Crit }}(Q, \gamma, \nu),
\end{array}\right.
$$
for some   $\gamma \in\left(0,\frac{Q}{2}\right)$, where  the positive constant $C$ is independent of $\varepsilon.$ 



Apart from the introduction,  the organization of the paper is as follows.
 In Section \ref{sec2}, we recall some basics of the Fourier analysis on the graded Lie group $\mathbb{G}$ to make the paper self-contained.
 In Section \ref{sec3}, we will clarify our strategy to prove the main results of this paper.
 Section \ref{sec4} is devoted to proving the main results of this article. In particular, using the Fourier analysis on the graded Lie group $\mathbb{G}$,  in the framework of Sobolev spaces of negative order, we prove that $p_{\text{Crit}}(Q, \gamma, \nu)=1+\frac{2\nu}{Q+2\gamma}$ is the new critical exponent for some  $\gamma\in (0, \frac{Q}{2})$ by proving a global-in-time existence of small data Sobolev solutions for $p>p_{\text{Crit}}(Q, \gamma, \nu)$ and a finite time blow-up of weak solutions for $1<p<p_{\text{Crit}}(Q, \gamma, \nu)$ to (\ref{eq0010}). We also derive a sharp estimate for the lifespan in the sub-critical case: $1<p<p_{\text{Crit}}(Q, \gamma, \nu)$.


%
%

\section{Preliminaries: Analysis on    graded Lie  groups} \label{sec2}
For more details on the material of this section, we refer to \cite{Folland, Fischer, RF17}.
\subsection{Graded Lie groups} A graded Lie group $\mathbb{G}$  is a connected and simply connected  Lie group whose Lie algebra $\mathfrak{g}$ is {\it graded}, that is, $\mathfrak{g}$ admits a vector space decomposition  $\mathfrak{g}= \bigoplus_{i=1}^\infty \mathfrak{g}_i,$ for which all but finitely many $\mathfrak{g}_i$'s are $\{0\}$ and satisfy the identities $[\mathfrak{g}_i, \mathfrak{g}_j] \subset \mathfrak{g}_{i+j}$ for all $i, j \in \mathbb{N}.$ Such a decomposition of a Lie algebra $\mathfrak{g}$ is called {\it gradation} of $\mathfrak{g}.$ A graded Lie algebra is {\it stratifiable} if there exists a gradation of $\mathfrak{g}$ such that $[\mathfrak{g}_1, \mathfrak{g}_i ]= \mathfrak{g}_{i+1}$ for all $i \in \mathbb{N}.$  A Lie group $\mathbb{G}$ whose Lie algebra $\g$ is stratifiable is called a {stratified Lie group}.  This immediately shows that every stratified Lie group is graded.  The Heisenberg group, more generally, $H$-type groups, Engel groups and Cartan groups are examples of graded Lie groups.

We define a family of dilations $D_r^{\mathfrak{g}},\,r>0,$ on a Lie algebra $\mathfrak{g}\cong \mathbb{R}^n $ as the vector space automorphisms of $\mathfrak{g}$ of the form $D_r^{\mathfrak{g}}:=\exp(\ln(r)A)$ for some diagonalisable matrix $A \sim \text{diag}[\nu_1, \nu_2, \ldots, \nu_n]$ with positive eigenvalues $0<\nu_1 \leq  \nu_2\leq \ldots \leq \nu_n$ on $\mathfrak{g}$ such that   $$D_r^{\mathfrak{g}}[X, Y]=[D_r^{\mathfrak{g}}X, D_r^{\mathfrak{g}}Y],$$
for all $X, Y \in \mathfrak{g}$ and $r>0.$ The  positive eigenvalues $0<\nu_1 \leq  \nu_2\leq \ldots \leq \nu_n$ of $A$ are called the dilations' weights of $\g.$  A Lie algebra $\g \cong \mathbb{R}^n$ is called {\it homogeneous} if there exists a family of dilations $D_r^{\mathfrak{g}},\,r>0,$ on $\g.$ It is well known that the existence of a family of dilations on $\g$ implies that $\g$ is a nilpotent Lie algebra. The Lie group $\mathbb{G}:=\exp{\g},$ which is connected and simply connected, is called  {\it homogeneous} if $\g$ is a homogeneous Lie algebra.  The family of dilations $\{D_r^{\mathfrak{g}}: r>0\}$ on Lie algebra $\mathfrak{g}$ induces a family of dilations $\{D_r: r>0\}$ on the group $\mathbb{G}$ by $D_r:=\exp \circ D_r^{\mathfrak{g}} \circ \exp^{-1},\,\, r>0.$

It is easy to note that for a given  graded Lie algebra $\mathfrak{g}= \bigoplus_{i=1}^\infty \mathfrak{g}_i$, the sequence of subspaces $\mathfrak{I}_k:=\bigoplus_{i=k}^\infty \mathfrak{g}_i$ forms a finite nested sequence of ideals in $\g.$ Thus, using these ideals $\mathfrak{I}_k,$ any basis $\{X_1, X_2, \ldots, X_n\}$ given as the union of the bases $\{X_1, X_2, \ldots, X_{n_i}\}$ is necessarily a strong Malcev basis of $\g.$ Such a basis of a graded Lie algebra $\g$ gives rise to a  family of dilations $D_r^{\g}, \,r>0,$ on $\g$ using the matrix given by  $A X_j=i X_j$ for every $X_j \in \g_i,$ that is, $D_r^{\g} X_j=r^{i}X_j.$ 

We may identify $\mathbb{G}$ with $\mathbb{R}^{n}$ with $n=\text{dim}~\mathfrak{g}$ via the exponential map $\exp: \mathfrak{g} \rightarrow G$ given by $x=\exp(x_1X_1+x_2X_2+\cdots+x_nX_n) \in \mathbb{G}$ having a basis of $\g.$ Using this identification, we can naturally identify a function on $\mathbb{G}$ as a function on $\mathbb{R}^n.$ These exponential coordinates allow us to represent the action of $D_r$ on $\mathbb{G}$ explicitly:
$$D_r(x)=rx:=\exp(r^{\nu_1}X_1+r^{\nu_2}X_2+\cdots+r^{\nu_n}X_n)=(r^{\nu_1}x_1, r^{\nu_2}x_2, \ldots, r^{\nu_n}x_n),$$
for $x:=(x_1,x_2,\ldots,x_n) \in \mathbb{G},\,\,r>0.$ This notion of dilation on $\mathbb{G}$ is crucial to define the notion of homogeneity for functions, measures, and operators. For examples, the bi-invariant Haar measure $dx$ on $\mathbb{G},$ which is just a Lebesgue measure on $\mathbb{R}^n,$ is $Q$-homogeneous in the sense that 
$$d(D_r(x))=r^Q dx,$$ where $Q:= \sum_{i=1}^\infty i \dim
\g_i=\nu_1+\nu_2+\ldots+\nu_n$ is called the homogeneous dimension of $\mathbb{G}.$ It is customary to jointly rescale weights so that $\nu_1=1.$ This also shows that $Q \geq n.$
Identifying the elements of $\g$ with the left-invariant vector fields, each $X_j$ is a homogeneous differential operator of degree $\nu_j.$ For every multi-index $\alpha \in \mathbb{N}^n_0$, we set $X^\alpha= X_1^{\alpha_1} X_2^{\alpha_2}\ldots X_r^{\alpha_n}$ in the universal enveloping algebra $\mathfrak{U}(\mathfrak{g})$ of the Lie algebra $\mathfrak{g}.$ Then $X^\alpha$ is of homogeneous degree $[\alpha]:=\alpha_1 \nu_1+\alpha_2 \nu_2+\ldots+\alpha_n \nu_n.$

\subsection{Homogeneous quasi-norms on homogeneous  groups and polar decomposition} 
A {\it homogeneous quasi-norm} on a homogeneous group $\G$ is a continuous function $|\cdot|:\G \rightarrow [0, \infty)$ such that it satisfies the following properties:
\begin{itemize}
\item $|x|=0$ if and only if $x=e_\G.$
\item  $|x^{-1}|=|x|$
\item $|D_rx|=r|x|$ for $r>0.$
\end{itemize}
There always exists a homogeneous quasi-norm in any homogeneous group $\mathbb{G}$ (\cite{Folland}). One can show the existence of a homogeneous quasi-norm on $\mathbb{G}$, which is $C^\infty$-smooth on $\mathbb{G} \backslash \{e_\mathbb{G}\}.$ Every homogeneous quasi-norm satisfies the following triangle inequality with the constant $C\geq 1:$
$$|xy| \leq C (|x|+|y|)\quad \forall x, y \in \mathbb{G}.$$
In fact, it is always possible to choose a  homogeneous quasi-norm on any homogeneous group that satisfies the triangle inequality with constant $C=1.$ Any two homogeneous quasi-norms on $\mathbb{G}$ are equivalent.

Similar to the Euclidean space, there is a notion of polar decomposition on a homogeneous group $\mathbb{G}$ for the homogenous quasi-norm $|\cdot|.$  Let 
$$\mathfrak{S}:={x \in G:\quad |x|=1}$$ be the unite sphere with respect to the homogenous quasi-norm $|\cdot|.$ Then there exists a unique Radon measure $\sigma$ on $\mathfrak{S}$ such that for all $f \in L^1(\G),$ we have 
\begin{equation} \label{polardeco}
\int_{G} f(x)dx= \int_0^\infty \int_{\mathfrak{S}} f(ry) r^{Q-1} d\sigma(y)\, dr.
\end{equation}

\subsection{Positive Rockland operators on graded Lie groups} \label{Rocksec}

Now it is time to introduce the main object of the discussion, namely, Rockland operators. To define them, we first need to fix some notation for the continuous unitary representations of the group. Let $(\pi, \mathcal{H}_\pi)$ be a continuous unitary representation of a graded Lie group $\G.$ Denote the set of equivalence classes of all strongly continuous unitary representations of $\G$ by $\widehat{\G}.$ Here, the Hilbert space $\mathcal{H}_\pi$ denotes the representation space of $\pi.$ We also denote the space of all smooth vectors of $\pi$ by $\mathcal{H}_\pi^\infty,$ which is a subspace of $\mathcal{H}_\pi.$
The infinitesimal representation of the Lie algebra $\mathfrak{g}$ and its extension to the universal enveloping Lie algebra $\mathfrak{U}(\mathfrak{g})$ will also be denoted by $\pi.$ We note here that the space of left-invariant vector fields and the algebra of left-invariant differential operators on $\G$ can be identified with $\mathfrak{g}$ and $\mathfrak{U}(\mathfrak{g}),$ respectively. For a left-invariant differential operator $T$, let us denote by $\pi(T),$ the infinitesimal representation $d\pi(T)$ associated with $\pi \in \widehat{\G}.$ 

A left-invariant differential operator $\mathcal{R}$ on a homogeneous group $\G$ called the {\it Rockland operator} if it is homogeneous of positive degree $\nu,$ that is, 
$$\mathcal{R}(f \circ D_r)=r^{\nu} (\mathcal{R}f) \circ D_r, \quad r>0,\,\, f \in C^\infty(\G), $$ and 
the operator $\pi(\mathcal{R})$ is injective on $\mathcal{H}^\infty_\pi$ for every nontrivial representation $\pi \in \widehat{\G},$ that is, 
\begin{equation}\label{Rock}
\forall v \in \mathcal{H}^\infty_\pi \quad \pi(\mathcal{R})v=0 \implies v=0.
\end{equation}
The condition \eqref{Rock} is known as the Rockland condition. The Rockland condition for $\mathcal{R}$ is equivalent to the {\it hypoellipticity} of $\mathcal{R},$ that is locally, $\mathcal{R}f \in C^\infty(\G) \implies f \in C^\infty(\G).$ This equivalence is commonly known as the {\it Rockland conjecture} (see \cite{Rock}) and was resolved in \cite{HN} (see also \cite{Miller}).
A Rockland operator is {\it positive} when 
$$\int_{\G} \mathcal{R}f(x) \overline{f(x)} dx \geq 0,\quad \forall f \in \mathcal{S}(\G).$$ 

It is a celebrated result of ter Elst and Robinson \cite{terRob97} which says that if there is a Rockland operator on a homogeneous Lie group $\G$, then the group $\G$ must be graded.
On the other hand, an infinite family of positive Rockland operators can be created for any graded Lie group $\G.$ Indeed, the operators given by 
\begin{equation}
\mathcal{R}:= \sum_{j=1}^n(-1)^{\frac{\nu_0}{\nu_j}} a_j X_j^{2 \frac{\nu_0}{\nu_j}}, \quad \text{with}\,\, a_1, a_2, \ldots, a_n>0
\end{equation}
for any strong Malcev basis $\{X_1, X_2, \ldots, X_n\}$ of the Lie algebra $\g$ and any common multiple $\nu_0$ of $\nu_1, \nu_2, \ldots, \nu_n,$ are positive Rockland operators of homogeneous degree $\nu=2\nu_0.$ It is easy to see that if $\mathcal{R}$ is a positive Rockland operator, then its powers $\mathcal{R}^k,\, k \in \mathbb{N},$ and complex conjugate $\overline{\mathcal{R}}$ are also Rockland operators.  

{\it Throughout this paper, we will always assume that a Rockland operator is always positive and essentially self-adjoint on $L^2(\G).$}

In the stratified case, assume that $\{X_1, X_2, \ldots, X_{n_1}\}$ is a basis of the first stratum $\g_1$ of the stratified Lie algebra. Then any left-invariant {\it sub-Laplacian} (with analyst sign convention) on $\G$
\begin{align}\label{stratified}
\mathcal{L}_\G:=-(X_1^2+X_2^2+\cdots+X_{n_1}^2)
\end{align}is a positive Rockland operator of the homogeneous degree $\nu=2.$ 
On $\g =\mathbb{R}^n,$ with the trivial stratification and canonical family of dilation $D_r(x)=rx, \,\,r>0,$ on the group $(\mathbb{R}^n, +)$, the Laplace operator operator $-\Delta_x:=-\sum_{i=1}^n \partial_{x_i}^2$ is a particular case of a positive sub-Laplacian. By equipping the group $\G=(\mathbb{R}^n, +)$ with another  isotropic or anisotropic family of dilations with the dilations' weights $\mathbb{R}^n \ni (\nu_1, \nu_2,\ldots, \nu_n) \neq (1, 1, \ldots, 1)$  determined by the canonical basis of $\g =\mathbb{R}^n,$ the operator 
\begin{equation}
\mathcal{R}:= \sum_{j=1}^n(-1)^{\frac{\nu_0}{\nu_j}} a_j \partial_j^{2 \frac{\nu_0}{\nu_j}}, \quad \text{with}\,\, a_1, a_2, \ldots, a_n>0
\end{equation}
is a positive Rockland operator of Homogeneous degree $\nu:=2\nu_0$ on $\G=(\mathbb{R}^n, +)$ provided  $\nu_0$ is any common multiple of $\nu_1, \nu_2, \ldots, \nu_n.$

\subsection{Fourier transform on graded Lie groups}
 
One of the important tools to deal with PDEs on graded Lie groups is the operator-valued group Fourier transform on $\G$. The group Fourier transform $\mathcal{F}_{\G   }(f)(\pi):\mathcal{H}_\pi \rightarrow \mathcal{H}_\pi$ of $f\in \mathcal{S}(\G)\cong \mathcal{S}(\mathbb{R}^n), $  at $\pi\in\widehat{\G},$ is a linear mapping that can be represented by an infinite matrix once we choose a basis for
the Hilbert space $\mathcal{H}_\pi,$  and defined by 
\begin{equation} \label{gft}
\mathcal{F}_{\G   }(f)(\pi)=\widehat{f}(\pi):=\int_{\G}f(x)\pi(x)^*dx = \int_{\G} f(x)\pi(x^{-1})\,dx.
\end{equation}
For $f \in L^2(\G),$ the operator $\widehat{f}(\pi)$ is a Hilbert-Schmidt operator on $\mathcal{H}_{\pi}$ for each $\pi \in \widehat{G}.$ Moreover, there exists a measure $\mu$ on $\widehat{G}$ such that the following inversion formula 
$$f(x)= \int_{\widehat{\G}} \text{Tr}(\pi(x)\widehat{f}(\pi)) d\mu(\pi)$$
holds for every $f \in \mathcal{S}(\G)$ and $x \in \G.$

Additionally, the following {\it Plancherel identity} is also true for $f \in \mathcal{S}(\G):$
\begin{equation}
\int_{\G} |f(x)|^2 \,dx =\int_{\widehat{G}} \|\widehat{f}(\pi)\|_{\text{HS}(\mathcal{H}_\pi)}^2\, d\mu(\pi).
\end{equation}
Furthermore, the Fourier transform $\mathcal{F}_{\G   }$ extends uniquely to a unitary isomorphism from $L^2(\G)$ onto the space $L^2(\widehat{\G}),$ where the space $L^2(\widehat{\G})$ is defined as the direct integral of Hilbert spaces of measurable fields of operators 
$$L^2(\widehat{\G}):=\int_{\widehat{G}}^{\oplus} \text{HS}(\mathcal{H}_\pi) d\mu(x)$$
with the norm 
$$\|\tau\|_{L^2(\widehat{\G})}= \left( \int_{\widehat{\G}} \|\tau_\pi\|_{\text{HS}(\mathcal{H}_\pi)}^2\, d\mu(\pi) \right)^{\frac{1}{2}}.$$
The measure $\mu$  is called the {Plancherel measure} on $\widehat{\G}$.

 Moreover, for any $f \in L^2(\mathbb{G})$,  we have 
$$
\mathcal{F}_{\mathbb G}(\mathcal{R} f)(\pi)=\pi(\mathcal{R})\widehat{f}(\pi).
$$
The authors in \cite{spectrum}  proved that the spectrum of the operator $\pi(\mathcal{R})$ with $\pi\in \widehat{ \mathbb{G}}\backslash \{1\}$, is discrete and lies in $(0, \infty)$.  Thus  we can choose an orthonormal basis for $\mathcal{H}_\pi$  such that the infinite matrix associated
to the self-adjoint operator $\pi(\mathcal{R})$ has the  following  representation 
\begin{align}\label{matrix}
\pi(\mathcal{R})=\left(\begin{array}{cccc}
\pi_1^2 & 0 & \cdots & \cdots \\
0 & \pi_2^2 & 0 & \cdots \\
\vdots & 0 & \ddots & \\
\vdots & \vdots & & \ddots
\end{array}\right)
\end{align}
where $\pi_i, \, i=1,2,\ldots,$ are strictly positive real numbers and  $\pi\in \widehat {\mathbb{G}}\backslash \{1\}$. 

\subsection{Sobolev spaces on graded Lie groups and interpolation inequalities}
The Sobolev spaces on graded Lie groups were systematically studied by Fischer and the third author in \cite{RF17, Fischer}. 

The inhomogeneous Sobolev spaces $H ^s(\G):=H^s_{\mathcal{R}}(\G), s \in \mathbb{R}$, associated to  positive Rockland operator  $\mathcal{R}$ of homogeneous degree $\nu$, is defined  as
$$
H ^s\left(\mathbb{G}\right):=\left\{f \in \mathcal{D}^{\prime}\left(\mathbb{G}\right):(I+\mathcal{R})^{s / \nu} f \in L^2\left(\mathbb{G}\right)\right\},
$$
with the norm $$\|f\|_{H ^s\left(\mathbb{G}\right)}:=\left\|(I+\mathcal{R})^{s / \nu} f\right\|_{L^2\left(\mathbb{G}\right)}.
$$  
Similarly, we define the homogeneous Sobolev space  $ \dot{H}^{p, s}_{\mathcal{R}}(\mathbb{G}):=\dot{H}^{p, s}(\mathbb{G})$ on $\G$  as the space of all  $f\in \mathcal{D}'(\mathbb{G})$ such that 
$\mathcal{R}^{{s}/{\nu}}f\in L^p(\mathbb{G})$ with the norm
$$\|f\|_{\dot{H}^{p, s}(\mathbb{G})}:=\left\|\mathcal{R}^{s / \nu} f\right\|_{L^p\left(\mathbb{G}\right)}.
$$  
The reason for omitting the subscript $\mathcal{R}$ is that 	
these Sobolev spaces are independent of the choice of a Rockland operator $\mathcal{R}$.    

Let  $\mathbb  G$ be a graded Lie group with homogeneous dimension $Q$.	  Then we use the following inequalities developed in \cite{30,Fischer, RF17}  throughout this work.

\begin{itemize}
\item  {\bf Hardy-Littlewood-Sobolev inequality} \cite{Fischer, RF17}:   Let $s>0$ and $1<p<q<\infty$ be such that 	$$ 	\frac{s}{Q}=\frac{1}{p}-\frac{1}{q}.$$ Then   \begin{align}\label{eq177}
	\|f\|_{{L}^{q}(
		\mathbb G)}  \lesssim  \|f\|_{{\dot H} ^{p,s}(\mathbb G)}\simeq   \|\mathcal{R}^{\frac{s}{\nu}}f\|_{L^p(\mathbb G)}. 	\end{align}
\item {\bf Gagliardo-Nirenberg  inequality} \cite{30}: 
Let 
$				s\in(0,1], 1<r<\frac{Q}{s},\text { and }~  2 \leq q \leq \frac{rQ}{Q-sr} 
.$	Then  
\begin{align}\label{eq16}
	\|u\|_{L^q(\mathbb G)} \lesssim\|u\|_{{\dot H} ^{r,s}(\mathbb G)}^\theta\|u\|_{L^2(\mathbb G)}^{1-\theta},
\end{align}
for $\theta=\left(\frac{1}{2}-\frac{1}{q}\right)/{\left(\frac{s}{Q}+\frac{1}{2}-\frac{1}{r}\right)}\in[0,1]$, provided that $\frac{s}{Q}+\frac{1}{2}\neq \frac{1}{r}$. 

\end{itemize}

\section{Philosophy of our approach}\label{sec3}
In this section, we will describe our approach to prove the main results of this paper.  First, we provide the notion of mild solutions to (\ref{eq0010}) in our framework. Consider the inhomogeneous system 
\begin{align} \label{eq001000}
\begin{cases}
	u_{tt}+\mathcal{R}u +u_{t} =F(t,x), & x\in  \mathbb{G},~t>0,\\
	u(0,x)=  u_0(x),  & x\in  \mathbb{G},\\ u_t(0, x)=  u_1(x), & x\in  \mathbb{G}.
\end{cases}
\end{align}
By applying Duhamel’s principle, the solution to the above system can be written as 
$$u(t, x)=   u_{0} *  E_{0}(t, x)+  u_{1} *  E_{1}(t, x)+\int_{0}^{t}F(s, x) *  E_{1}(t-s, x) \;ds,$$
where   $*$ denotes the  group convolution product on $ \mathbb G$ with respect to the $x$ variable, and  $E_{0}$ and $E_{1}$  represent   the propagators to (\ref{eq001000}) in  the homogeneous case $F=0$   with initial data $\left(u_{0}, u_{1}\right)=\left(\delta_{0}, 0\right)$ and $\left(u_{0}, u_{1}\right)=$ $\left(0, \delta_{0}\right)$, respectively. Our main interest lies in considering power-type nonlinearities, that is, $F(t,x)=|u(t, x)|^p$, and this will always be the case throughout the paper.

A function   $u$ is  said to be a {\it mild solution} to (\ref{eq0010})  on $[0, T]$ if $u$ is a fixed point for  the    integral operator $N: u \in X_s(T) \mapsto N u(t, x) ,$ given  by 
\begin{align}\label{f2inr} 
N u(t, x):=u^{\text{lin}}(t, x) + u^{\text{non}}(t, x),
\end{align}	in the energy evolution space $X_s(T) \doteq \mathcal{C}\left([0, T], H^{s}(\mathbb{G})\right), s\in (0, 1]$, 
equipped with the norm
\begin{align} \label{Evolution norm}
\|u\|_{X_s(T)}&:=\sup\limits_{t\in[0,T]}\left ( (1+t)^{\frac{\gamma}{\nu}} \|u(t,\cdot)\|_{L^2}+(1+t)^{\frac{s+\gamma}{\nu}}\|u(t,\cdot)\|_{ \dot H^s}\right )
\end{align}
with $\gamma>0,$
where   $$u^{\text{lin}}(t, x)=   u_{0} *  E_{0}(t, x)+  u_{1} *  E_{1}(t, x)$$  
 is the solution to the corresponding linear Cauchy problem (\ref{eq0010}), and 
$$ u^{\text{non}}(t, x)= \int_{0}^{t} F(s, x) *  E_{1}(t-s, x)  \;ds.$$
We aim to prove the global-in-time existence and uniqueness of small data Sobolev solutions of low regularity to the semilinear damped wave equation (\ref{eq0010}) with
the help of the Banach's  fixed point theorem argument.  We find a  unique fixed point (say) $u^*$ of the operator $N$, that means, $u^*=N u^* \in X_s(T)$ for all positive $T$. More precisely, to find such a unique fixed point,   we will establish two crucial inequalities of the form
\begin{align}\label{Banach1}
\|N u\|_{X_s(T)} & \lesssim  \left\|\left(u_0, u_1\right)\right\|_{\mathcal{A}^s}+\|u\|_{X_s(T)}^p,  \end{align}
and \begin{align}\label{Banach2}
\|N u-N v\|_{X_s(T)} & \lesssim\|u-v\|_{X_s(T)}\left[  \|u\|_{X_s(T)}^{p-1}+\|v\|_{X_s(T)}^{p-1}\right],
\end{align}
for any $u, v \in X_s(T)$ with initial data space $\mathcal{A}^s:=\ (H^s \cap \dot{H}^{-\gamma}\ ) \times\ (L^2 \cap \dot{H}^{-\gamma}\ )$. 
We will consider sufficiently small $\left\|\left(u_0, u_1\right)\right\|_{\mathcal{A}^s}<\varepsilon $ so that combining (\ref{Banach1}) with (\ref{Banach2}) we can apply  Banach's fixed point theorem to ensure that there exists a global-in-time small data unique Sobolev solution $u^*=Nu^* \in X_s(T)$ for all  $T>0$, which also gives the solution to (\ref{eq0010}). Here we want to note that the treatment of the power-type nonlinear term in $\dot {H}^{-\gamma}$ is based on the applications of the Hardy-Littlewood-Sobolev inequality and the  Gagliardo-Nirenberg inequality on the graded Lie group $\mathbb G$.

Before considering the semilinear Cauchy problem, first, we will determine the $\dot {H}^s$-norm estimate for the solution to the linear problem (\ref{Linear-system}) via the group Fourier transform on the graded Lie group $\mathbb G$ concerning the spatial variable.


On the other hand,   to prove a  blow-up (in-time) result in the subcritical case to the Cauchy problem  (\ref{eq0010}) for certain ranges of $p$ regardless of the size of the initial data, we need to introduce a suitable notion of a weak solution to the   Cauchy problem (\ref{eq0010}).

For any $T>0$, 	 a  weak solution of the Cauchy problem (\ref{eq0010}) in $[0, T) \times \mathbb G$ is a function $u \in L_{\text {loc }}^p\left([0, T) \times \mathbb G\right)$ that   satisfies  the following  integral relation:
\begin{align}\label{2999}\nonumber
&	\int_0^T \int_{\mathbb G} u(t, x)\left(\partial_t^2 \phi(t, x)-\mathcal{R} \phi(t, x)-\partial_t \phi(t, x)\right) d x \;d t -\varepsilon\int_{\mathbb{H}^n} u_0(x) \phi(0, x) d x\\&-\varepsilon\int_{\mathbb{H}^n} u_1(x) \phi(0, x) d x+\varepsilon\int_{\mathbb G} u_0(x) \partial_t \phi(0, x) d x =	\int_0^T \int_{\mathbb G}|u(t, x)|^p \phi(t, x) d x \;d t,
\end{align}
for any $\phi \in \mathcal{C}_{0}^{\infty}([0, T) \times \mathbb G)$ and   any $t \in(0, T)$.
If $T=\infty$, we call $u$ to be a global-in-time weak solution to (\ref{eq0010}), otherwise   $u$ is said to be a local-in-time weak solution to (\ref{eq0010}). 
We obtain blow-up of weak solutions even for small
data in the case  $1 < p < p_{\text{Crit}}(Q, \gamma, \nu)$ by employing the test function method. 




\section{Main results}\label{sec4}
This section is devoted to present the proofs of the main results of this paper. In particular, in the setting of negative order Sobolev spaces, we show that  $p_{\text{Crit}}(Q, \gamma, \nu)=1+\frac{2\nu}{Q+2\gamma}$ is the new critical exponent for some  $\gamma\in (0, \frac{Q}{2})$  to the Cauchy problem (\ref{eq0010}) by proving global-in-time existence of small data Sobolev solutions of lower regularity for $p>p_{\text{Crit}}(Q, \gamma, \nu) $  and a finite time blow-up of weak solutions for $1<p<p_{\text{Crit}}(Q, \gamma, \nu)$. We begin this section with the following  $ \dot {H}^s$- estimates for solutions to the homogeneous Cauchy problem (\ref{Linear-system}), which will be crucial to proving the global-in-time existence of small data Sobolev solutions. 
 
\begin{proof}[Proof of Theorem \ref{Linear}] 
Applying the group Fourier transform \eqref{gft} on $\mathbb{G}$ to the linear system   (\ref{Linear-system})  with respect to $x$,  for all $\pi \in \widehat{\mathbb{G}}$,   we get   a Cauchy problem related to a parameter-dependent functional differential equation for $ \widehat{u}(t, \pi)$ as follows:
\begin{align}\label{eq6661}
	\begin{cases}
		\partial^2_t\widehat{u}(t, \pi)+\pi(\mathcal{R}) \widehat{u}(t, \pi) +\partial_t \widehat{u}(t, \pi)=0,&  \pi \in\widehat{\mathbb{G}},~t>0,\\ \widehat{u}(0, \pi)=\widehat{u}_0( \pi), & \pi \in\widehat{\mathbb{G}},\\ \partial_t\widehat{u}(0, \pi)=\widehat{u}_1( \pi), & \pi \in\widehat{\mathbb{G}},
	\end{cases} 
\end{align}
where  $\pi(\mathcal{R})$ is the symbol of the Rockland operator $\mathcal{R}$ on ${\mathbb{G}}$. For $ m, k \in \mathbb{N}$,   we introduce the notation
\begin{align}\label{basis}
	\widehat{u}(t,  \pi)_{m, k} \doteq\left(\widehat{u}(t,  \pi) e_k, e_{m}\right)_{\mathcal{H}_\pi},
\end{align}
where $ \{e_m\}_{m\in \mathbb{N}}$ is  the same orthonormal basis in the representation space $\mathcal{H}_\pi$ that gives   us  (\ref{matrix}). 

Then $\widehat{u}(t,  \pi)_{m, k}$ solves the following  infinite system of  ordinary differential equation with respect to $t$    variable 
\begin{align}\label{eqq7}
	\begin{cases}
		\partial^2_t\widehat{u}(t, \pi)_{m, k}+	\partial_t\widehat{u}(t, \pi)_{m, k}+ \beta_{m, \pi}^{2 }  \widehat{u}(t, \pi)_{m, k}= 0,&  \pi \in\widehat{\mathbb{G}},~t>0,\\ \widehat{u}(0, \pi)_{m, k}=\widehat{u}_0( \pi)_{m, k}, & \pi \in\widehat{\mathbb{G}},\\ \partial_t\widehat{u}(0, \pi)_{m, k}=\widehat{u}_1( \pi)_{m, k}, & \pi \in\widehat{\mathbb{G}},
	\end{cases}
\end{align}
where  we denote $\beta_{m, \pi}^{2 } = \pi_m^2.$ 	 The characteristic equation  of the above system is given by
\[\lambda^2+	\lambda+\beta_{m, \pi}^{2 } =0.\]
Consequently,  the characteristic roots    are    given by  $$\lambda_1=\frac{-1- \sqrt{1-4\beta_{m, \pi}^{2 }}}{2}  \quad \text{ and}\quad \lambda_2=\frac{-1+ \sqrt{1-4\beta_{m, \pi}^{2 }}}{2}.$$
We consider the following cases:
\begin{itemize}
	\item  When $|\beta_{m, \pi}|<\delta  \ll 1$:   \begin{align}\label{leq 1}\nonumber
		&  \lambda_1=\frac{-1- \sqrt{1-4\beta_{m, \pi}^{2 }}}{2}
		=\frac{-1- (1-4{\beta_{m, \pi}}^2 )^{\frac{1}{2}}}{2}=-1+\mathcal{O}\left(\beta_{m, \pi}^{2 }\right),\\
		&  \lambda_{2}= \frac{-1+ \sqrt{1-4\beta_{m, \pi}^{2 }}}{2}=\frac{-1+(1-4\beta_{m, \pi}^2 )^{\frac{1}{2}}}{2}=-\beta_{m, \pi}^{2 }+\mathcal{O}\left(\beta_{m, \pi}^4\right).
	\end{align}
	
	\item When $|\beta_{m, \pi}|>N \gg 1$: \begin{align*} 
		\lambda_{1}= \frac{-1-(1-4\beta_{m, \pi}^2 )^{\frac{1}{2}}}{2}  &=\frac{-1-2i|\beta_{m, \pi}| (1-(4\beta_{m, \pi}^2)^{-1})^{\frac{1}{2}}} {2}  
		\\&=  -\frac{1}{2}-i|\beta_{m, \pi}|+\mathcal{O}\left(|\beta_{m, \pi}|^{-1}\right), \end{align*} 
  \begin{align}\label{leq 11}
		\lambda_{2}= \frac{-1+(1-4\beta_{m, \pi}^2 )^{\frac{1}{2}}}{2}  =  -\frac{1}{2}+i|\beta_{m, \pi}|+\mathcal{O}\left(|\beta_{m, \pi}|^{-1}\right).\end{align}
	\item When $\delta<|\beta_{m, \pi}|<N$: \begin{align}\label{leq 111}
		\text{Re}( \lambda_{1})<0\quad \text{ and}\quad  \text{Re}( \lambda_{2})<0.
	\end{align} 
\end{itemize}
Thus, the solution to the   homogeneous  system   (\ref{eqq7}) is given by
\begin{align}\label{eqq77} 
	\widehat{u}(t, \pi)_{m, k}&= K_0(t,  \pi)_{m,k} \widehat{u}_0( \pi)_{m, k}+ K_1(t,  \pi)_{m,k}\widehat{u}_1( \pi)_{m,k},
\end{align}
where from (\ref{leq 1}), (\ref{leq 11}), and (\ref{leq 111}), we can write
\begin{align}\label{shyam} 
	& 	K_0(t,  \pi)_{m,k}=\frac{\lambda_1{e}^{\lambda_2 t}-\lambda_2{e}^{\lambda_1 t}}{\lambda_1-\lambda_2} \nonumber \\\\&=\begin{cases}
		\frac{\left(-1+\mathcal{O}(\beta_{m, \pi}^2)\right){e}^{\left(-\beta_{m, \pi}^{2 }+\mathcal{O} (\beta_{m, \pi}^{4} )\right)t}-\left(-\beta_{m, \pi}^{2 }+\mathcal{O} (\beta_{m, \pi}^{4})\right){e}^{\left(-1 +\mathcal{O} (\beta_{m, \pi}^{2} )\right)t}}{-1+\mathcal{O}\left(\beta_{m, \pi}^{2 }\right)} & \text { for }|\beta_{m, \pi}|<\delta,  \\\\ 	\frac{\left(i|\beta_{m, \pi}|-\frac{1}{2}+\mathcal{O}\left(|\beta_{m, \pi}|^{-1}\right)\right)  {e}^{\left(-i|\beta_{m, \pi}|-\frac{1}{2}+\mathcal{O}\left(|\beta_{m, \pi}|^{-1}\right)\right) t}}{2 i|\beta_{m, \pi}|+\mathcal{O}(1)} &  \\\\
		\qquad 	-\frac{\left(-i|\beta_{m, \pi}|-\frac{1}{2}+\mathcal{O}\left(|\beta_{m, \pi}|^{-1}\right)\right)  {e}^{\left(i|\beta_{m, \pi}|-\frac{1}{2}+\mathcal{O}\left(|\beta_{m, \pi}|^{-1}\right)\right) t}}{2 i|\beta_{m, \pi}|+\mathcal{O}(1)}
		& \text { for }|\beta_{m, \pi}|>N .
	\end{cases} 
\end{align}

and
\begin{align} \label{shyam2}
	&\nonumber	K_1(t,  \pi)_{m,k}=\frac{ {e}^{\lambda_1 t}-{e}^{\lambda_2 t}}{\lambda_1-\lambda_2}\\\\&=\left\{\begin{array}{ll}
		\frac{ {e}^{\left(-1+\mathcal{O}\left(\beta_{m, \pi}^2\right)\right) t}- {e}^{\left(-\beta_{m, \pi}^2+\mathcal{O}\left(\beta_{m, \pi}^4\right)\right) t}}{-1+\mathcal{O}\left(\beta_{m, \pi}^2\right)} & \text { for }|\beta_{m, \pi}|<\delta, \\\\
		\frac{ {e}^{\left(i|\beta_{m, \pi}|-\frac{1}{2}+\mathcal{O}\left(|\beta_{m, \pi}|^{-1}\right)\right) t}- {e}^{\left(-i|\beta_{m, \pi}|-\frac{1}{2}+\mathcal{O}\left(|\beta_{m, \pi}|^{-1}\right)\right) t}}{2 i|\beta_{m, \pi}|+\mathcal{O}(1)}  & \text { for }|\beta_{m, \pi}|>N .
	\end{array}\right.
\end{align}
Thus from the above  asymptotic expression of eigenvalues, for each $m, k\in \mathbb N$ and for all $\pi \in \widehat{\mathbb{G}}$,  we have the following point-wise estimates for $K_0$ and $K_1$:
\begin{align}\label{K_0estimate}
	\left|K_0(t, \pi)_{m,k}\right| \lesssim\left\{\begin{array}{ll}
		|\beta_{m, \pi}|^2  {e}^{-c t}+ {e}^{-ct\beta_{m, \pi}^2 } & \text { for }|\beta_{m, \pi}|<\delta \ll 1, \\\\
		\mathrm{e}^{-c t} & \text { for } \delta \leq|\beta_{m, \pi}| \leq N, \\\\
		\mathrm{e}^{-c t} & \text { for }|\beta_{m, \pi}|>N \gg 1,
	\end{array}\right.
\end{align}
and 
\begin{align}\label{K_1estimate}
	\left|K_1(t, \pi)_{m,k}\right| \lesssim\left\{\begin{array}{ll}
		{e}^{-c t}+ {e}^{-ct\beta_{m, \pi}^2 } & \text { for }|\beta_{m, \pi}|<\delta \ll 1, \\\\
		{e}^{-c t} & \text { for } \delta \leq|\beta_{m, \pi}| \leq N, \\\\
		|\beta_{m, \pi}|^{-1}  {e}^{-c t} & \text { for }|\beta_{m, \pi}|>N \gg 1,
	\end{array}\right.
\end{align}
for some   constant $c>0$.  Before we estimate $\|u(t, \cdot )\|_{\dot H^s }$,  for $ |\beta_{m,\pi}|<\delta$ we notice that 	\begin{align}\label{decay}
	\left| \beta_{m,\pi}^{\frac{4(s+\gamma)}{\nu}}     {e}^{-c\beta_{m,\pi}^2 t }\right|\lesssim (1+t)^{-\frac{2(s+\gamma)}{\nu}} 
\end{align}
for $s+\gamma \geq 0.$  To estimate (\ref{decay}), we used the fact that   $(\delta^2+y)^{\frac{s+\gamma}{\nu}} e^{-cy}$   is uniformly bounded for large $y$.

Now  using the Plancherel formula and using  the notation,  we obtain 
\begin{align}\label{eq01}\nonumber
	&\|u(t, \cdot )\|_{\dot H^s }^2=\int_{\widehat{\mathbb{G}}}\|\pi(\mathcal{R})^{\frac{s}{\nu}} \hat{u}(t,  \pi)\|_{\text{HS}}^2  \,d\mu( \pi)\\\nonumber
	&\qquad=\int_{\widehat{\mathbb{G}}} \sum_{m, k\in \mathbb{N}} \pi_m^{\frac{4s}{\nu}} |\hat{u}(t,  \pi)_{m,k}|^2  \,d\mu( \pi)\\\nonumber
	&\qquad\lesssim \int_{\widehat{\mathbb{G}}} \sum_{m, k\in \mathbb{N}}\pi_m^{\frac{4s}{\nu}} \left[  |K_0(t,  \pi)_{m,k}|^2  |\widehat{u}_0( \pi)_{m, k}|^2 + |K_1(t,  \pi)_{m,k}|^2 |\widehat{u}_1( \pi)_{m,k} |^2 \right] d\mu( \pi)\\\nonumber
 	&\qquad =   \sum_{m, k\in \mathbb{N}} \int_{\widehat{\mathbb{G}}} \pi_m^{\frac{4s}{\nu}} \left[  |K_0(t,  \pi)_{m,k}|^2  |\widehat{u}_0( \pi)_{m, k}|^2 + |K_1(t,  \pi)_{m,k}|^2 |\widehat{u}_1( \pi)_{m,k} |^2 \right] d\mu( \pi)\\
	&\qquad= I_{K_0}+I_{K_1},
\end{align}
where we switched integral and the series  due to the fact that each term is non-negative and  we denote
$$I_{K_0}:= \sum_{m, k\in \mathbb{N}} \int_{\widehat{\mathbb{G}}} \pi_m^{\frac{4s}{\nu}}   |K_0(t,  \pi)_{m,k}|^2  |\widehat{u}_0( \pi)_{m, k}|^2   d\mu( \pi)$$
and 
$$I_{K_1}:= \sum_{m, k\in \mathbb{N}} \int_{\widehat{\mathbb{G}}} \pi_m^{\frac{4s}{\nu}}  |K_1(t,  \pi)_{m,k}|^2 |\widehat{u}_1( \pi)_{m,k} |^2  d\mu( \pi).$$
When $|\beta_{m, \pi}|<\delta \ll 1$,  keeping in mind the notation    $\beta_{m, \pi}^{2 } = \pi_m^2$, using \eqref{K_0estimate}, \eqref{K_1estimate}, and \eqref{decay},  we obtain the following estimates: 
\begin{align}\label{eq02}\nonumber
	I_{K_0}&= \sum_{m, k\in \mathbb{N}} \int_{\{\pi\in \widehat{\mathbb{G}}:|\beta_{m, \pi}|<\delta  \}} \pi_m^{\frac{4(s+\gamma)}{\nu}}   |K_0(t,  \pi)_{m,k}|^2  \beta_{m,\pi}^{-\frac{4\gamma}{\nu}}|\widehat{u}_0( \pi)_{m, k}|^2   d\mu( \pi)\\\nonumber
	&	 \lesssim   \sum_{m, k\in \mathbb{N}} \int_{\{\pi\in \widehat{\mathbb{G}}:|\beta_{m, \pi}|<\delta  \}} \beta_{m,\pi}^{\frac{4(s+\gamma)}{\nu}}   \{	|\beta_{m, \pi}|^2  {e}^{-c t}+ {e}^{-ct\beta_{m, \pi}^2 }\}^2 \beta_{m,\pi}^{-\frac{4\gamma}{\nu}}|\widehat{u}_0( \pi)_{m, k}|^2   d\mu( \pi)\\\nonumber
	&	=   \sum_{m, k\in \mathbb{N}} \int_{\{\pi\in \widehat{\mathbb{G}}:|\beta_{m, \pi}|<\delta  \}} \beta_{m,\pi}^{\frac{4(s+\gamma)}{\nu}}  {e}^{-c2t\beta_{m, \pi}^2 }  \{	|\beta_{m, \pi}|^2  {e}^{-c t(1-\beta_{m, \pi}^2 )}+ 1 \}^2 \beta_{m,\pi}^{-\frac{4\gamma}{\nu}}|\widehat{u}_0( \pi)_{m, k}|^2   d\mu( \pi)\\\nonumber
	& \lesssim (1+t)^{-\frac{2(s+\gamma)}{\nu}}	    \sum_{m, k\in \mathbb{N}} \int_{\{\pi\in \widehat{\mathbb{G}}:|\beta_{m, \pi}|<\delta  \}}   \beta_{m,\pi}^{-\frac{4\gamma}{\nu}}|\widehat{u}_0( \pi)_{m, k}|^2   d\mu( \pi)\\
	& \lesssim (1+t)^{-\frac{2(s+\gamma)}{\nu}}	  \|{u}_0 \|_{\dot H^{-\gamma}}^2,
\end{align}
and  
\begin{align}\label{eq03}\nonumber
	I_{K_1}&= \sum_{m, k\in \mathbb{N}} \int_{\{\pi\in \widehat{\mathbb{G}}:|\beta_{m, \pi}|<\delta  \}} \pi_m^{\frac{4s}{\nu}}  |K_1(t,  \pi)_{m,k}|^2 |\widehat{u}_1( \pi)_{m,k} |^2  d\mu( \pi)\\\nonumber
	&	\leq   \sum_{m, k\in \mathbb{N}} \int_{\{\pi\in \widehat{\mathbb{G}}:|\beta_{m, \pi}|<\delta  \}} \beta_{m,\pi}^{\frac{4(s+\gamma)}{\nu}}   \{	   {e}^{-c t}+ {e}^{-ct\beta_{m, \pi}^2 }\}^2 \beta_{m,\pi}^{-\frac{4\gamma}{\nu}}|\widehat{u}_0( \pi)_{m, k}|^2   d\mu( \pi)\\\nonumber
	&	\leq   \sum_{m, k\in \mathbb{N}} \int_{\{\pi\in \widehat{\mathbb{G}}:|\beta_{m, \pi}|<\delta  \}} \beta_{m,\pi}^{\frac{4(s+\gamma)}{\nu}} {e}^{-2ct\beta_{m, \pi}^2 }  \{	   {e}^{-c t(1-\beta_{m, \pi}^2)}+  1\}^2 \beta_{m,\pi}^{-\frac{4\gamma}{\nu}}|\widehat{u}_0( \pi)_{m, k}|^2   d\mu( \pi)\\\nonumber
	& \lesssim (1+t)^{-\frac{2(s+\gamma)}{\nu}}	    \sum_{m, k\in \mathbb{N}} \int_{\{\pi\in \widehat{\mathbb{G}}:|\beta_{m, \pi}|<\delta  \}}   \beta_{m,\pi}^{-\frac{4\gamma}{\nu}}|\widehat{u}_1( \pi)_{m, k}|^2   d\mu( \pi)\\
	&\lesssim (1+t)^{-\frac{2(s+\gamma)}{\nu}}	  \|{u}_1 \|_{ \dot H^{-\gamma}}^2   .
\end{align}
 
Now consider the case  $|\beta_{m, \pi}|>N  \gg 1$. Using   the estimates \eqref{K_0estimate} and  \eqref{K_1estimate}, we find that
\begin{align}\label{eq04}\nonumber
	I_{K_0}&= \sum_{m, k\in \mathbb{N}} \int_{\{\pi\in \widehat{\mathbb{G}}:|\beta_{m, \pi}|>N \}} \pi_m^{\frac{4s}{\nu}}   |K_0(t,  \pi)_{m,k}|^2  |\widehat{u}_0( \pi)_{m, k}|^2   d\mu( \pi)\\\nonumber
	& 	\lesssim  {e}^{-2c t}   \sum_{m, k\in \mathbb{N}} \int_{\{\pi\in \widehat{\mathbb{G}}:|\beta_{m, \pi}|>N \}}  \beta_{m,\pi}^{\frac{4s}{\nu}}    |\widehat{u}_0( \pi)_{m, k}|^2   d\mu( \pi)\\
	& \lesssim 
 {e}^{-2c t}   \|{u}_0 \|_{H ^{-s}}^2,
\end{align}
and 
\begin{align}\label{eq05}\nonumber
	I_{K_1}&= \sum_{m, k\in \mathbb{N}} \int_{\{\pi\in \widehat{\mathbb{G}}:|\beta_{m, \pi}|>N \}} \pi_m^{\frac{4s}{\nu}}  |K_1(t,  \pi)_{m,k}|^2 |\widehat{u}_1( \pi)_{m,k} |^2  d\mu( \pi)\\\nonumber
	&\lesssim    {e}^{-2c t}    \sum_{m, k\in \mathbb{N}} \int_{\{\pi\in \widehat{\mathbb{G}}:|\beta_{m, \pi}|>N \}} \beta_{m,\pi}^{\frac{4s}{\nu}}  	|\beta_{m, \pi}|^{-2}  |\widehat{u}_1( \pi)_{m,k} |^2  d\mu( \pi)\\\nonumber
	&=   {e}^{-2c t}  \sum_{m, k\in \mathbb{N}} \int_{\{\pi\in \widehat{\mathbb{G}}:|\beta_{m, \pi}|>N \}}   \left(  { \beta_{m,\pi}^{2}} \right)^{\frac{2s}{\nu}-1} 	    ~|\widehat{u}_1( \pi)_{m,k} |^2  d\mu( \pi)\\\nonumber
	&=   {e}^{-2c t}  \sum_{m, k\in \mathbb{N}} \int_{\{\pi\in \widehat{\mathbb{G}}:|\beta_{m, \pi}|>N \}}  (1+\beta_{m,\pi}^2)^{\frac{2(s-1)}{\nu}}  \left( \frac{ \beta_{m,\pi}^{2}}{1+  \beta_{m,\pi}^{2}}\right)^{\frac{2s}{\nu}-1} 	  \frac{ 1}{    \left(1+  \beta_{m,\pi}^{2}\right)^{1-\frac{2}{\nu}} } 	 ~|\widehat{u}_1( \pi)_{m,k} |^2  d\mu( \pi)\\
	&\lesssim     {e}^{-2c t}  \|{u}_1 \|_{H ^{s-1}}^2,
\end{align}
where we used     $\nu\geq 2$  to get the last  inequality.  

Now we consider the  final case when   $\delta<|\beta_{m, \pi}|<N$. Again  from the estimates \eqref{K_0estimate}  and  \eqref{K_1estimate}, we get
\begin{align}\label{eqs01}\nonumber
	I_{K_0}&= \sum_{m, k\in \mathbb{N}} \int_{\{\pi\in \widehat{\mathbb{G}}:\delta\leq |\beta_{m, \pi}|\leq N \}}\pi_m^{\frac{4s}{\nu}}   |K_0(t,  \pi)_{m,k}|^2  |\widehat{u}_0( \pi)_{m, k}|^2   d\mu( \pi)\\\nonumber
	& \lesssim {e}^{-2c t}  \sum_{m, k\in \mathbb{N}} \int_{\{\pi\in \widehat{\mathbb{G}}:\delta \leq |\beta_{m, \pi}|\leq N \}}\beta_{m,\pi}^{\frac{4s}{\nu}}    |\widehat{u}_0( \pi)_{m, k}|^2   d\mu( \pi)\\\nonumber
	& \lesssim {e}^{-2c t}  \sum_{m, k\in \mathbb{N}} \int_{\{\pi\in \widehat{\mathbb{G}}:\delta\leq |\beta_{m, \pi}|\leq N \}}     (1+\beta_{m,\pi}^2)^{\frac{2s}{\nu}}    
	|\widehat{u}_0( \pi)_{m, k}|^2   d\mu( \pi)\\& \lesssim {e}^{-2c t} \|{u}_0 \|_{H ^{s}}^2,
\end{align}
and 
\begin{align}\label{eqs05}\nonumber
	I_{K_1}&= \sum_{m, k\in \mathbb{N}} \int_{\{\pi\in \widehat{\mathbb{G}}:\delta\leq |\beta_{m, \pi}|\leq N \}}\pi_m^{\frac{4s}{\nu}}  |K_1(t,  \pi)_{m,k}|^2 |\widehat{u}_1( \pi)_{m,k} |^2  d\mu( \pi)\\\nonumber
	& \lesssim {e}^{-2c t}   \sum_{m, k\in \mathbb{N}} \int_{\pi\in \widehat{\mathbb{G}}:\delta\leq |\beta_{m, \pi}|\leq N \}}(1+\beta_{m,\pi}^2)^{\frac{2s}{\nu}}    |\widehat{u}_0( \pi)_{m, k}|^2   d\mu( \pi)\\\nonumber
	& = {e}^{-2c t}  \sum_{m, k\in \mathbb{N}} \int_{\{\pi\in \widehat{\mathbb{G}}:\delta \leq |\beta_{m, \pi}|\leq N \}}(1+\beta_{m,\pi}^2)^{\frac{2(s-1)}{\nu}}   (1+\beta_{m,\pi}^2)^{\frac{2}{\nu}}      |\widehat{u}_0( \pi)_{m, k}|^2   d\mu( \pi) 	\\
	&\lesssim {e}^{-2c t} \|{u}_1 \|_{H ^{s-1}}^2.
\end{align}
 
 Combining   all the cases for $|\beta_{m, \pi}|$, that is,     (\eqref{eq02},  \eqref{eq03}),   (\eqref{eq04},  \eqref{eq05}), and (\eqref{eqs01},   \eqref{eqs05}) along with \eqref{eq01}, we  obtain    
$ \dot {H} ^s$-decay estimate for the solution to linear system  \eqref{Linear-system}  as  $$
\|u(t, \cdot)\|_{ \dot{H} ^s} \lesssim(1+t)^{-\frac{s+\gamma}{\nu}}\left(\left\|u_0\right\|_{H^s  \cap \dot{H}^{-\gamma} }+\left\|u_1\right\|_{H ^{s-1} \cap \dot{H} ^{-\gamma}}\right)
$$		for any $t\geq 0$.
\end{proof}
\subsection{Global existence result} 		 
Denoting $p_{\text {Crit }}(Q, \gamma, \nu):=1+\frac{2\nu}{Q+2\gamma}$, we have the following global in-time well-posedness result for small data solutions to the Cauchy problem (\ref{eq0010}) in the energy evolution space  $\mathcal C\left([0,T],  H^s(\mathbb{G})\right) $.  
\begin{proof}[Proof of Theorem \ref{well-posed}]
Our main aim is to prove the following 	two crucial inequalities 
\begin{align}\label{to prove1}
	\|N u\|_{X_s(T)} & \lesssim\left\|\left(u_{0}, u_{1}\right)\right\|_{ \mathcal{A}^s}+\|u\|_{X_s(T)}^{p},\\\label{to prove2}		\|N u-N v\|_{X_s(T)} &\lesssim \|u-v\|_{X_s(T)}\left(\|u\|_{X_s(T)}^{p-1}+\|v\|_{X_s(T)}^{p-1}\right).
\end{align}  
Now from   the estimate (\ref{hom}) of Theorem \ref{Linear}, for  $s \geq 0$ and $\gamma \in \mathbb{R}$ such that $s+\gamma \geq  0$,  we recall that 
\begin{equation} \label{4er}
	\|u(t, \cdot)\|_{\dot{{H}} ^s(\HH)} \lesssim(1+t)^{-\frac{s+\gamma}{\nu}}\left(\left\|u_0\right\|_{H^s  \cap \dot {{H}} ^{-\gamma} }+\left\|u_1\right\|_{H ^{s-1}  \cap \dot{{H}} ^{-\gamma} }\right) .
\end{equation}
Now depending on the value of $\gamma $, from \eqref{4er},  	using  the Sobolev embedding   $L^2\subset H ^{s-1}$ for $s\leq 1$ and $ H ^{s} \subset L^2$ for $s\geq0$, we have the following   crucial estimates for Sobolev solutions to the linear Cauchy problem.
\begin{itemize}
	\item For   $\gamma=0$ and  $s \in [0, 1]$, we get 
	\begin{align}\label{eq15}\nonumber
		\|u(t, \cdot)\|_{{\dot{H}} ^s}&
		\lesssim(1+t)^{-\frac{s}{\nu}}\left(\left\|u_0\right\|_{H^s  \cap L^2 }+\left\|u_1\right\|_{H ^{s-1}  \cap L^2}\right) \\ &\lesssim(1+t)^{-\frac{s}{\nu}}\left(\left\|u_0\right\|_{H ^s }+\left\|u_1\right\|_{L^2}\right).
	\end{align}
	\item 	For $\gamma>0$ and  $s \in [0, 1]$, 	 	we have 
	\begin{align}\label{eq14}
		\|u(t, \cdot)\|_{\dot{{H}} ^s} \lesssim(1+t)^{-\frac{s+\gamma}{\nu}}\left(\left\|u_0\right\|_{ H^s \cap {\dot{H}} ^{-\gamma}}+\left\|u_1\right\|_{L^2 \cap {\dot{H}} ^{-\gamma}}\right).
	\end{align}
	
\end{itemize}
In particular,  for $s=0$ in  (\ref{eq14}), using  the Sobolev embedding $ H ^{s} \subset L^2$ for $s\geq0$, we obtain \begin{align}\label{eq1444}
	\nonumber	\|u(t, \cdot)\|_{L^2} &\lesssim(1+t)^{-\frac{\gamma}{\nu}}\left(\left\|u_0\right\|_{ L^2 \cap {\dot{H}} ^{-\gamma}}+\left\|u_1\right\|_{L^2 \cap {\dot{H}} ^{-\gamma}}\right) \\&\lesssim(1+t)^{-\frac{\gamma}{\nu}}\left(\left\|u_0\right\|_{ H^s \cap {\dot{H}} ^{-\gamma}}+\left\|u_1\right\|_{L^2 \cap {\dot{H}} ^{-\gamma}}\right).
\end{align}
Now from (\ref{eq14}) and   (\ref{eq1444}), for   $  s\in (0, 1]$ and $\gamma>0$,   we can write
\begin{align*}
	&	(1+t)^{\frac{s+\gamma}{\nu}}\|u(t, \cdot)\|_{{\dot {H}} ^s} +(1+t)^{\frac{\gamma}{\nu}}\|u(t, \cdot)\|_{L^2}\\ &\quad \lesssim \left(\left\|u_0\right\|_{H^s \cap {\dot{H}} ^{-\gamma}}+\left\|u_1\right\|_{L^2 \cap {\dot {H}} ^{-\gamma}}\right)=\left\|\left(u_{0}, u_{1}\right)\right\|_{ \mathcal{A}^s}.
\end{align*}
Thus from above, using the notation (\ref{Evolution norm}), we can   claim     
$u \in X_s(T)$ such that  \begin{align}\label{2number100}
	\|u^{\text{lin}}\|_{X_s(T)} \lesssim   \left\|\left(u_{0}, u_{1}\right)\right\|_{ \mathcal{A}^s}.
\end{align} 
Now under some conditions for $p$, we show that  
$$	\|u^{\text{non}}\|_{X_s(T)} \lesssim    \left\| u\right\|_{ X_s(T)}^p.$$
First we   calculate $L^2$    and $\dot H^{-\gamma} $ norm   of $|u(t, \cdot)|^p$ 	 in order to  estimate    $\|u^{\text{non}}\|_{X_s(T)} $.   

\textbf{\underline{$L^2$  norm   of $|u(t, \cdot)|^p$}:} Observe that  $\left \||u(\kappa, \cdot)|^p\right\|_{L^2}=	\left \|u(\kappa, \cdot)\right\|_{L^{2p}}^p$.  Now applying the Gagliardo-Nirenberg  inequality (\ref{eq16}) for $\left \|u(\kappa, \cdot)\right\|_{L^{2p}}$, we get
\begin{align}\label{eqq21}\nonumber
	&	\left \||u(\kappa, \cdot)|^p\right\|_{L^2}=	\left \|u(\kappa, \cdot)\right\|_{L^{2p}}^p\\\nonumber
	&\lesssim\|u(\kappa, \cdot)\|_{\dot {H}^s}^{p\theta} \|u(\kappa, \cdot)\|_{L^2 }^{p(1-\theta)}\\\nonumber
	&= (1+\kappa)^{-\frac{p}{\nu}\left[ s\theta+\gamma \right ]}  \left\{ (1+\kappa)^{\frac{(s+\gamma)}{\nu} } \|u(\kappa, \cdot)\|_{\dot {H}^s} \right\}^{p\theta}  \left\{ (1+\kappa)^{\frac{\gamma}{\nu} } \|u(\kappa, \cdot)\|_{L^2}\right\}^{p(1-\theta)}\\\nonumber
	&= (1+\kappa)^{-\frac{p}{\nu}\left[ \gamma+\frac{Q}{2}\left(1-\frac{1}{p}\right) \right ]}  \left\{ (1+\kappa)^{\frac{(s+\gamma)}{\nu} } \|u(\kappa, \cdot)\|_{\dot {H}^s } \right\}^{p\theta}  \left\{ (1+\kappa)^{\frac{\gamma}{\nu} } \|u(\kappa, \cdot)\|_{L^2 }\right\}^{p(1-\theta)}\\
	&\lesssim (1+\kappa)^{-\frac{p}{\nu}\left(  \gamma+\frac{Q}{2} \right ) +\frac{Q}{2\nu}}    \left\| u\right\|_{ X_s(T)}^p,
\end{align}
for $\kappa\in [0, T]$ with $\theta=\frac{Q}{2s}(1-\frac{1}{p})\in [0, 1]$. Using the fact that  $\theta\in [0,1]$, we have the restriction on    $p$ as  
\begin{align}\label{eq18}
	1\leq p\leq \frac{Q}{Q-2s} \quad \text{if} ~~Q>2s.
\end{align}
\textbf{\underline{$\dot H^{-\gamma} $   norm   of $|u(t, \cdot)|^p$}:}	 From the    Hardy-Littlewood-Sobolev inequality (\ref{eq177}), we obtain
\begin{align}\label{Hgamma}
	\left \||u(\kappa, \cdot)|^p\right\|_{\dot H^{-\gamma} }
	\lesssim	\left \| |u(\kappa, \cdot)|^p\right\|_{L^{m}}
	=	\left \| u(\kappa, \cdot)\right\|_{L^{mp}}^p,
\end{align}
where  $\frac{1}{m}-\frac{1}{2}=\frac{\gamma}{Q}$ provided that  $0<\gamma <Q$ and $1<m<2$. Also, from the fact  that    $m \in(1,2)$, we have to restrict    $ 0<\gamma  <\frac{Q}{2}.$    Now  applying  the Gagliardo-Nirenberg inequality (\ref{eq16}) for $\left \| u(\kappa, \cdot)\right\|_{L^{mp}}$,  from \eqref{Hgamma}, we obtain
\begin{align}\label{eq220}\nonumber
	&	\left \||u(\kappa, \cdot)|^p\right\|_{\dot H^{-\gamma} } \lesssim\| u(\kappa, \cdot)\|_{\dot {H}^s(\mathbb{G})}^{p\theta_1} \|u(\kappa, \cdot)\|_{L^2(\mathbb{G})}^{p(1-\theta_1)}\\\nonumber
	&= (1+\kappa)^{-\frac{p}{\nu} ( s\theta_1+\gamma)   }  \left\{ (1+\kappa)^{\frac{(s+\gamma)}{\nu} } \|u(\kappa, \cdot)\|_{\dot {H}^s(\mathbb{G})} \right\}^{p\theta_1}  \left\{ (1+\kappa)^{\frac{\gamma}{\nu} } \|u(\kappa, \cdot)\|_{L^2(\mathbb{G})}\right\}^{p(1-\theta_1)}\\\nonumber
	&\lesssim (1+\kappa)^{-\frac{p}{\nu} \left[  {Q}(\frac{1}{2}-\frac{1}{mp})+\gamma\right]  } \left\| u\right\|_{ X_s(T)}^p\\ 
	&\lesssim (1+\kappa)^{-\frac{p}{\nu} ( {\gamma}+\frac{Q}{2} )+  \frac{1}{\nu}\left(  {\gamma}+\frac{Q}{2}\right)  }  \left\| u\right\|_{ X_s(T)}^p,
\end{align}
for $\kappa\in [0, T]$ provided  $\theta_1=\frac{Q}{s}(\frac{1}{2}-\frac{1}{mp})\in [0, 1]$.  Since $\theta_1\in [0, 1]$, we  have another restriction $p$ as 
\begin{align}\label{eq17}\nonumber
	& \frac{2}{m} \leq p\leq   \frac{2Q}{m(Q-2s)}\\\implies&     	1+\frac{2\gamma}{Q} \leq p  
	\leq   \frac{Q+2\gamma}{Q-2s}\quad  \text { if } Q>2s,
\end{align}
where we used  $m=\frac{2Q}{Q+2\gamma} $. In conclusion,  from   (\ref{eq18}) and (\ref{eq17}), if we consider 
\begin{align*}
	1+\frac{2\gamma}{Q} \leq p  \left\{\begin{array}{ll}
		<\infty & \text { if } Q \leq 2s, \\
		\leq   \frac{Q}{Q-2s}& \text { if } Q>2s,
	\end{array}\right. 
\end{align*}
then from  the estimates  (\ref{eqq21}) and (\ref{eq220}),   it immediately  follows that   
\begin{align}\label{eq20}
	\left \||u(\kappa, \cdot)|^p\right\|_{L^2\cap \dot H^{-\gamma} }\lesssim (1+\kappa)^{-\frac{p}{\nu} (  {\gamma}+\frac{Q}{2} )+  \frac{1}{\nu}\left(  {\gamma}+\frac{Q}{2}\right)  } \left\| u\right\|_{ X_s(T)}^p
\end{align}
for     $\kappa\in [0, T]$. In order to   estimate the solution itself in $L^2$, we first see that 
\begin{align}  \label{mk}
	\left\|u^{\mathrm{non}}(t, \cdot)\right\|_{L^2}&=\bigg \|  \int\limits_0^t  |u(\kappa,\cdot )|^p* E_1(t-\kappa, \cdot) d\kappa\bigg\|_{L^2 }\nonumber\\
	&\leq  \int\limits_0^t \| |u(\kappa,\cdot )|^p*  E_1(t-\kappa, \cdot)\|_{L^2 }d\kappa. 
\end{align}
Next, note that   $\tilde{u}(t,x):=(|u(\kappa,\cdot)|^p*  E_1(t, \cdot))(x)$ is also a solution to the linear system (\ref{Linear-system}) with  initial data $u_0=0$, $u_1= |u(\kappa, \cdot)|^p$. To estimate    Duhamel’s term on the interval  $[0, t]$,  we   apply  $ (L^2 \cap \dot {H}^{-\gamma} )-L^2$ estimate  (\ref{eq1444}) for  $\tilde{u}$ in \eqref{mk} to get
\begin{align}\nonumber
	\left\|u^{\mathrm{non}}(t, \cdot)\right\|_{L^2}
	&\lesssim   \int\limits_0^t (1+t-\kappa)^{-\frac{\gamma}{\nu}}  \| |u(\kappa,\cdot)|^p\|_{L^2 \cap \dot H^{-\gamma} }  d\kappa\\\nonumber
	&\lesssim \int_0^t(1+t-\kappa)^{-\frac{\gamma}{\nu}}  (1+\kappa)^{-\frac{p}{\nu} (  {\gamma}+\frac{Q}{2} )+  \frac{1}{\nu}\left(  {\gamma}+\frac{Q}{2}\right)  }  d\kappa ~  \left\| u\right\|_{ X_s(T)}^p\\\nonumber
	& \lesssim(1+t)^{-\frac{\gamma}{\nu}} \int_0^{\frac{t}{2}}(1+\kappa)^{-\frac{p}{\nu} (  {\gamma}+\frac{Q}{2} )+  \frac{1}{\nu}\left(  {\gamma}+\frac{Q}{2}\right)  } d\kappa ~  \left\| u\right\|_{ X_s(T)}^p\\\label{eq222222}
	& \qquad +(1+t)^{-\frac{p}{\nu} (  {\gamma}+\frac{Q}{2} )+  \frac{1}{\nu}\left(  {\gamma}+\frac{Q}{2}\right)  }  \int_{\frac{t}{2}}^t(1+t-\kappa)^{-\frac{\gamma}{\nu}} \mathrm{~d} \kappa\|u\|_{X_s(T)}^p,
\end{align}
where we used the  estimate      (\ref{eq20})  for  the nonlinear term. 

Now we estimate separately the two integrals on the right-hand side of (\ref{eq222222}).  Notice  that the first  integral  in (\ref{eq222222})    $$\int_0^{\frac{t}{2}}(1+\kappa)^{-\frac{p}{\nu} (  {\gamma}+\frac{Q}{2} )+  \frac{1}{\nu}\left(  {\gamma}+\frac{Q}{2}\right)  } d\kappa $$  converges  uniformly  over $\left[0, \frac{t}{2}\right]$ for   $p>1+\frac{2\nu}{Q+2 \gamma}$. For the integral over $[\frac{t}{2}, t]$, we see that    
$$
\int_{\frac{t}{2}}^t(1+t-\kappa)^{-\frac{\gamma}{\nu}} \mathrm{~d} \kappa \lesssim\left\{\begin{array}{ll}
	(1+t)^{1-\frac{\gamma}{\nu}} & \text { if } \gamma<\nu, \\
	\ln (\mathrm{e}+t) & \text { if } \gamma=\nu, \\
	1 & \text { if } \gamma>\nu.
\end{array}\right.
$$
\begin{itemize}
	\item  Considering $p>1+\frac{2\nu}{Q+2 \gamma}$ if $\gamma \leq \nu$, we   see that 
	\begin{align}\label{gamma<nu}\nonumber
		&	(1+t)^{-\frac{p}{\nu} (  {\gamma}+\frac{Q}{2} )+  \frac{1}{\nu}\left(  {\gamma}+\frac{Q}{2}\right)  }  \int_{\frac{t}{2}}^t(1+t-\kappa)^{-\frac{\gamma}{\nu}}  {~d} \kappa\\\nonumber
		&\leq 	(1+t)^{-1}  \int_{\frac{t}{2}}^t(1+t-\kappa)^{-\frac{\gamma}{\nu}}  {~d} \kappa \\\nonumber
		&\lesssim 	(1+t)^{-1} \times  \left\{\begin{array}{ll}
			(1+t)^{1-\frac{\gamma}{\nu}} & \text { if } \gamma<\nu, \\
			\ln (\mathrm{e}+t) & \text { if } \gamma=\nu, 
		\end{array}\right. \\
		&\lesssim(1+t)^{-\frac{\gamma}{\nu}} . 
	\end{align}
	\item Considering  $p>1+\frac{2 \gamma}{Q+2 \gamma}$ if $\gamma>\nu$,  we get 
	\begin{align}\label{gamma>nu}\nonumber
		&(1+t)^{-\frac{p}{\nu} (  {\gamma}+\frac{Q}{2} )+  \frac{1}{\nu}\left(  {\gamma}+\frac{Q}{2}\right)  }  \int_{\frac{t}{2}}^t(1+t-\kappa)^{-\frac{\gamma}{\nu}}  {~d} \kappa\\
		&\leq 	(1+t)^{-\frac{\gamma}{\nu}}   \int_{\frac{t}{2}}^t(1+t-\kappa)^{-\frac{\gamma}{\nu}}  {~d} \kappa   \lesssim(1+t)^{-\frac{\gamma}{\nu}} .
	\end{align}
\end{itemize}
Thus from \eqref{eq222222}, \eqref{gamma<nu}, and \eqref{gamma>nu}, we finally get 
\begin{align}\label{eq23}
	(1+t)^{\frac{\gamma}{\nu}} 	\|u^{\text{non}}(t, \cdot )\|_{L^2} \lesssim    \left\| u\right\|_{ X_s(T)}^p. \end{align}
Again, in order to   estimate the solution itself in  $\dot {H}^s$, using     $\left(L^2 \cap \dot {H}^{-\gamma}\right)-\dot 
{H}^s$ estimate  (\ref{eq14}) for the interval $\left[0, \frac{t}{2}\right]$ and the $L^2-\dot {H}^s$ estimate (\ref{eq15}) for  the interval  $\left[\frac{t}{2}, t\right]$, we get 
\begin{align}\nonumber
	\left\|u^{\mathrm{non}}(t, \cdot)\right\|_{\dot {H}^s}
	&\leq   \int\limits_0^t \|    |u(\kappa,\cdot )|^p* E_1(t-\kappa, \cdot)  \|_{\dot {H}^s} d\kappa\\\nonumber
	& \lesssim \int_0^{\frac{t}{2}}(1+t-\kappa)^{-\frac{(s+\gamma)}{\nu}} (1+\kappa)^{-\frac{p}{\nu} (  {\gamma}+\frac{Q}{2} )+  \frac{1}{\nu}\left(  {\gamma}+\frac{Q}{2}\right)  }  d\kappa ~  \left\| u\right\|_{ X_s(T)}^p\\\nonumber
	& \qquad + \int_{\frac{t}{2}}^t (1+t-\kappa)^{-\frac{s}{\nu}}(1+\kappa)^{-\frac{p}{\nu} (  {\gamma}+\frac{Q}{2} )+  \frac{Q}{2\nu}  }   \mathrm{~d} \kappa\|u\|_{X_s(T)}^p \\\nonumber
	& \lesssim(1+t)^{-\frac{(s+\gamma)}{\nu}} \int_0^{\frac{t}{2}}(1+\kappa)^{-\frac{p}{\nu} (  {\gamma}+\frac{Q}{2} )+  \frac{1}{\nu}\left(  {\gamma}+\frac{Q}{2}\right)  } d\kappa ~  \left\| u\right\|_{ X_s(T)}^p\\\label{eq22}
	& \qquad +(1+t)^{-\frac{p}{\nu} (  {\gamma}+\frac{Q}{2} )+  \frac{Q}{2\nu}  }   \int_{\frac{t}{2}}^t(1+t-\kappa)^{-\frac{s}{\nu}} \mathrm{~d} \kappa\|u\|_{X_s(T)}^p .
\end{align}
For $s\in (0, 1]$ and  $p>1+\frac{2\nu}{Q+2 \gamma}$,
following the same computations   performed earlier, we can write the following:
\begin{align}\label{eqq22}
	(1+t)^{\frac{s+\gamma}{\nu}}\left\|u^{ \text{non}}(t, \cdot)\right\|_{\dot {H}^s} \lesssim\|u\|_{X_s(T)}^p .
\end{align}
Therefore, combining  (\ref{eq23}) and (\ref{eqq22}), it follows that
\begin{align}\label{LIN1}
	\left\|u^{ \text{non}}(t, \cdot)\right\|_{X_s(T)} \lesssim\|u\|_{X_s(T)}^p,
\end{align}
provided the exponent $p$  satisfies the   following  criteria:
\begin{center}
	$	1<	  p    \left\{\begin{array}{ll}
		<\infty & \text { if } Q \leq 2s, \\
		\leq   \frac{Q}{Q-2s}& \text { if } Q>2s,
	\end{array}\right. \quad $  and 
	$ \quad 
	p\begin{cases}
		>1+\frac{2\nu}{Q+2 \gamma}& \text{if} ~\gamma \leq \nu,\\
		>1+\frac{2 \gamma}{Q+2 \gamma}& \text{if $\gamma>\nu ,$}\\
		\geq 1+\frac{2 \gamma}{Q} . 
	\end{cases}$ 
\end{center} 
First, we examine the scenario when $\gamma>\nu$. Subsequently, its easy to see that 
$$
\max \left\{1+\frac{2 \gamma}{Q+2 \gamma}, 1+\frac{2 \gamma}{Q}\right\}=1+\frac{2 \gamma}{Q}.
$$
Next, we consider    $\gamma\leq \nu$.  In this case, we  have to compare between $1+\frac{2 \nu}{Q+2 \gamma}$  and $1+\frac{2 \gamma}{Q}$. Observe that the expressions  $1+\frac{2 \nu}{Q+2 \gamma}$  and $1+\frac{2 \gamma}{Q}$   intersect at a specific point   $\tilde \gamma$, which is nothing but  the positive root  of the quadratic equation $ 2\gamma^2+Q\gamma-\nu Q=0$.  Moreover, it is important to highlight  that the positive root  $\tilde  \gamma<\nu$ for any  $Q\geq 1.$  This implies that  
\begin{itemize}
	\item $\max \left\{1+\frac{2\nu}{Q+2 \gamma}, 1+\frac{2 \gamma}{Q} \right\}=1+\frac{2\nu}{Q+2 \gamma} $ for $\gamma \leq \tilde{\gamma}$,
	\item $\max \left\{1+\frac{2\nu}{Q+2 \gamma}, 1+\frac{2 \gamma}{Q} \right\}=1+\frac{2 \gamma}{Q} $ for  $\tilde{\gamma}<\gamma \leq \nu$.
\end{itemize}
Summarizing, for any $\gamma\in (0, \frac{Q}{2})$, the assumption  for the exponent $p$ is  reduced to    \begin{align*} 
	p\left\{\begin{array}{ll}
		>1+\frac{2\nu}{Q+2 \gamma} =p_{\text {Crit }}(Q, \gamma,\nu) & \text { if } \gamma \leq \tilde{\gamma}, \\
		\geq 1+\frac{2 \gamma}{Q} & \text { if } \gamma>\tilde{\gamma}.
	\end{array}\right. 
\end{align*}
Therefore, the condition for the exponent $p$  reduces to (\ref{eq24}).  Moreover, from (\ref{2number100}) and (\ref{LIN1}), we get our required first inequality (\ref{to prove1}), i.e., 
\begin{align}\label{eqlinera}
	\|N u\|_{X_s(T)} \lesssim  \left\|\left(u_{0}, u_{1}\right)\right\|_{ \mathcal{A}^s }+\|u\|_{X_s(T)}^{p}.
\end{align}
To establish (\ref{to prove2}), 	 our next objective is to   evaluate $
\|N u-N v\|_{X_s(T)}$.  First we notice that 
$$
\|N u-N v\|_{X_s(T)}=\left\|\int_0^t E_1(t-\kappa, \cdot) *\left(|u(\kappa, \cdot)|^p-|v(\kappa, \cdot)|^p\right)  {d} \kappa\right\|_{X_s(T)} .
$$
Now the   Hardy-Littlewood-Sobolev inequality  (\ref{eq177}) yields
\begin{align*}
	\left \| |u(\kappa, \cdot)|^p-|v(\kappa, \cdot)|^p \right\|_{\dot  H^{-\gamma} }
	\lesssim	\left \| |u(\kappa, \cdot)|^p-|v(\kappa, \cdot)|^p\right\|_{L^{m}}
\end{align*}
with $\frac{1}{m}-\frac{1}{2}=\frac{\gamma}{Q}$ provided    $\gamma \in (0,Q)$ and $m\in (1,2)$. 

Now using the fact that  $||u|^p-|v|^p|\leq p |u-v|(|u|^{p-1}+|v|^{p-1}) $ along with      H\"older's inequality on the right-hand side of the above inequality, we obtain 
\begin{align}\label{eq25}
	\left\||u(\kappa, \cdot)|^p-|v(\kappa, \cdot)|^p\right\|_{L^m} \lesssim\|u(\kappa, \cdot)-v(\kappa, \cdot)\|_{L^{m p}}\left(\|u(\kappa, \cdot)\|_{L^{m p}}^{p-1}+\|v(\kappa, \cdot)\|_{L^{m p}}^{p-1}\right) .
\end{align}
Our next aim is to  estimate three terms on the right-hand side of the previous inequality (\ref{eq25}).    Applying  the Gagliardo-Nirenberg inequality (\ref{eq16})   for   each   term that appeared on the right-hand side of (\ref{eq25}), we get 
\begin{align}\label{Final2}\nonumber
	& 	\left \| u(\kappa, \cdot)-v(\kappa, \cdot)\right\|_{L^{mp}}\\\nonumber
	&\lesssim\| u(\kappa, \cdot)-v(\kappa, \cdot)\|_{\dot {H}^s }^{\theta_1} \|u(\kappa, \cdot)\|_{L^2 }^{(1-\theta_1)}\\\nonumber
	&= (1+\kappa)^{-\frac{1}{\nu}\left(  {\gamma}+ {s\theta_1}\right)  }  \left\{ (1+\kappa)^{\frac{(s+\gamma)}{\nu} } \|u(\kappa, \cdot)-v(\kappa, \cdot)\|_{\dot {H}^s } \right\}^{\theta_1}   \left\{ (1+\kappa)^{\frac{\gamma}{\nu} } \|u(\kappa, \cdot)-v(\kappa, \cdot)\|_{L^2 }\right\}^{(1-\theta_1)}\\
	&\lesssim (1+\kappa)^{-\frac{1}{\nu}\left(  {\gamma}+ {s\theta_1}\right)  }\left\| u-v \right\|_{ X_s(T)}
\end{align}
and 
\begin{align}\label{Final1}\nonumber
	&	 	\left \| u(\kappa, \cdot)\right\|_{L^{mp}}^{p-1}\\\nonumber
	&\lesssim\| u(\kappa, \cdot)\|_{\dot {H}^s(\mathbb{H}^n)}^{(p-1)\theta_1} \|u(\kappa, \cdot)\|_{L^2(\mathbb{H}^n)}^{(p-1)(1-\theta_1)}\\\nonumber
	&= (1+\kappa)^{-\frac{(p-1)}{\nu}  (s\theta_1+\gamma)}  \left\{ (1+\kappa)^{\frac{(s+\gamma)}{\nu} } \|u(\kappa, \cdot)\|_{\dot {H}^s(\mathbb{H}^n)} \right\}^{(p-1)\theta_1}   \left\{ (1+\kappa)^{\frac{\gamma}{\nu} } \|u(\kappa, \cdot)\|_{L^2(\mathbb{H}^n)}\right\}^{(p-1)(1-\theta_1)}\\
	&\lesssim (1+\kappa)^{-\frac{(p-1)}{\nu}  (s\theta_1+\gamma)} \left\| u\right\|_{ X_s(T)}^{p-1},
\end{align}
for $\kappa\in [0, T]$ with $\theta_1=\frac{Q}{s}(\frac{1}{2}-\frac{1}{mp})\in [0, 1]$.  
Therefore, combining (\ref{eq25}), (\ref{Final2}), and (\ref{Final1}), we    obtain   
\begin{align*}
	\nonumber   & \left \| |u(\kappa, \cdot)|^p-|v(\kappa, \cdot)|^p \right\|_{\dot  H^{-\gamma}}  \\
	&\lesssim(1+\kappa)^{-\frac{1}{\nu}\left(  {\gamma}+ {s\theta_1}\right)  } (1+\kappa)^{-\frac{(p-1)}{\nu}  (s\theta_1+\gamma)} \left\| u-v \right\|_{ X_s(T)}   \left(\|u\|_{X_s(T)}^{p-1}+\|v\|_{X_s(T)}^{p-1}\right)\\
	& 
	\lesssim (1+\kappa)^{-\left(\frac{p-1}{\nu}+\frac{1}{\nu}\right) (\gamma+{Q}(\frac{1}{2}-\frac{1}{mp}))}  \left\| u-v \right\|_{ X_s(T)}   \left(\|u\|_{X_s(T)}^{p-1}+\|v\|_{X_s(T)}^{p-1}\right) \\
	&	= (1+\kappa)^{-\left(\frac{p-1}{\nu}+\frac{1}{\nu}\right) (\gamma+\frac{Q}{2}-\frac{1}{p}(\frac{Q}{2}+{\gamma}))}  \left\| u-v \right\|_{ X_s(T)}   \left(\|u\|_{X_s(T)}^{p-1}+\|v\|_{X_s(T)}^{p-1}\right) \\
	& 
	= (1+\kappa)^{- \frac{p}{\nu}(1-\frac{1}{p})(\gamma+\frac{Q}{2}) }  \left\| u-v \right\|_{ X_s(T)}   \left(\|u\|_{X_s(T)}^{p-1}+\|v\|_{X_s(T)}^{p-1}\right)\\
	&
	\lesssim(1+\kappa)^{- \frac{p}{\nu}\left( {\gamma}+\frac{Q}{2}\right) +   \frac{1}{\nu} ( {\gamma}+ \frac{Q}{2})}  \left\| u-v \right\|_{ X_s(T)}   \left(\|u\|_{X_s(T)}^{p-1}+\|v\|_{X_s(T)}^{p-1}\right). \nonumber
\end{align*}
Finally, using the given range of $p$, we see that $- \frac{p}{\nu}\left( {\gamma}+\frac{Q}{2}\right) +   \frac{1}{\nu} ( {\gamma}+ \frac{Q}{2})\leq 0$ and therefore
\begin{align}\label{Banach}
	\|N u-N v\|_{X_s(T)} \leq A \|u-v\|_{X_s(T)}\left(\|u\|_{X_s(T)}^{p-1}+\|v\|_{X_s(T)}^{p-1}\right),
\end{align}
for any $u, v \in X_s(T)$ and for some  positive constat $A.$
Also from (\ref{eqlinera}), we obtain 	
\begin{align}\label{Final banach}
	\|N u\|_{X_s(T)} \leq  B\left\|\left(u_{0}, u_{1}\right)\right\|_{ \mathcal{A}^s }+B\|u\|_{X_s(T)}^{p},
\end{align}  
for some positive constant $B$ with initial data space $\mathcal{A}^s : =(H^s\cap \dot  H^{-\gamma}) \times (L^2\cap \dot  H^{-\gamma})$. 
Consequently, for  some $r>1$,  we choose $R_0=rB\left\|\left(u_0, u_1\right)\right\|_{\mathcal{A}^s} $ with sufficiently small $\left\|\left(u_0, u_1\right)\right\|_{\mathcal{A}^s}<\varepsilon $  so that  
$$BR_0^p<\frac{R_0}{r} \quad \text{and} \quad 2AR_0^{p-1}<\frac{1}{r}.$$
Then     (\ref{Banach})   and (\ref{Final banach}) reduce to  
\begin{align}\label{Final banach1}
	\|N u\|_{X_s(T)} \leq  \frac{2R_0}{r} 
\end{align}  
and \begin{align}\label{Banach11}
	\|N u-N v\|_{X_s(T)} \leq  \frac{1}{r}\|u-v\|_{X_s(T)},			\end{align}
respectively for all $u, v\in \mathcal{B}(R_0):=\{u\in X_s(T): \|u\|_{X_s(T)}\leq R_0\}$.

Since   $\left\|\left(u_0, u_1\right)\right\|_{\mathcal{A}^s}<\varepsilon $ is sufficiently small,   (\ref{Final banach1}) implies that   $Nu \in X_s(T)$, that is, $N$ maps $X_s(T)$ into itself. 
On the other hand,   (\ref{Banach11}) implies that the map $N$ is a contraction mapping on the ball $\mathcal{B}(R_0)$ around the origin in the Banach space $X_s(T)$. Using  Banach's fixed point theorem,  we can say that there exists a uniquely determined fixed point $u^*$ of the operator $N$, which means $u^*=Nu^* \in X_s(T)$  for all positive $T$.  This implies that there exists a global-in-time small data Sobolev solution $u^*$ of the equation $ u^*=Nu^* $ in $ X_s(T)$, which also gives the solution to the semilinear damped wave equation (\ref{eq0010})  and this completes the proof of the theorem.  
\end{proof}

\subsection{Blow-up result} 
Now we prove blow-up of weak solutions in the subcritical case  $1<p<p(Q, \gamma, \nu )$ to the Cauchy problem (\ref{eq0010}).   
\begin{proof}[Proof of Theorem \ref{blow-up}] We apply the so-called test function method on the graded Lie group $\mathbb{G}$ in order to prove this result.   		For $R>1$, let us consider the  test function as
$$
\varphi_R(t, x)=\Phi\left(\frac{|x|}{R}\right) \Phi\left(\frac{t}{R^\nu}\right),
$$
where 	 $\Phi: \mathbb{R}_{+} \rightarrow[0,1]$ is the smooth cut-off function such that
$$
\Phi(r)=\left\{\begin{array}{ll}
	1, & 0 \leq r \leq 1 \\
	\searrow, & 1 \leq r \leq 2 \\
	0, & r \geq 2.
\end{array}\right.
$$
By changing variables $x=R \tilde{x}$ and $t=R^\nu \tilde{t}$ and the fact that $R>1$, we obtain the following   estimates on $\Omega_T:= (0, T) \times \mathbb{G},$ (see \cite{Yang})\vspace{0.2cm}
\begin{itemize}
	\item $\displaystyle \quad 
	\int_{\Omega_T}\left|\frac{\partial }{\partial t} \varphi_R(t, x)\right|^{\frac{p}{p-1}} \varphi_R(t, x)^{\frac{-1}{p-1}} d x d t \lesssim R^{-\frac{\nu p}{p-1}+Q+\nu},  
	$\vspace{0.5cm}
	\item $\displaystyle \quad 
	\int_{\Omega_T}\left|\mathcal{R} \varphi_R(t, x)\right|^{\frac{p}{p-1}} \varphi_R(t, x)^{\frac{-1}{p-1}} d x dt  \lesssim R^{-\frac{\nu p}{p-1}+Q+\nu},
	$\vspace{0.5cm}
	\item $\displaystyle \quad 
	\int_{\Omega_T}\left|\frac{\partial^2 }{\partial t^2}  \varphi_R(t, x)\right|^{\frac{p}{p-1}} \varphi_R(t, x)^{\frac{-1}{p-1}} d x d t \lesssim  R^{-\frac{2\nu p}{p-1}+Q+\nu}\leq  R^{-\frac{\nu p}{p-1}+Q+\nu}.
	$
\end{itemize}
For  $a, b \geq 0$,  using  $\varepsilon$-Young’s inequality
$$a b \leq \varepsilon a^r+C(\varepsilon) b^{r^{\prime}}, \quad \frac{1}{r}+\frac{1}{r^{\prime}}=1,$$
from the relation (\ref{2999}),   we obtain
\begin{align}\nonumber\label{eq30}
	&	\int_0^T \int_{\mathbb{G}}|u(t, x)|^p \varphi_R(t, x) \;d x \;d t+\varepsilon\int_{\mathbb{G}}\left(u_0(x)+u_1(x)\right) \varphi_R(0, x) \;d x\\\nonumber
	&\leq \int_0^T \int_{\mathbb{G}} |u(t, x)|\left(|\partial_t^2 \varphi_R(t, x)|+|\mathcal{R}\varphi_R(t, x)|+|\partial_t \varphi_R(t, x)\right)| \;d x \;d t\\\nonumber
	&=\int_0^T \int_{\mathbb{G}} \varphi_R(t, x)^{\frac{1}{p}}|u(t, x)| \left(|\partial_t^2 \varphi_R(t, x)|+|\mathcal{R}\varphi_R(t, x)|+|\partial_t \varphi_R(t, x)|\right) \varphi_R(t, x)^{-\frac{1}{p}} \;d x \;d t\\\nonumber
	&\leq \frac{1}{2}\int_0^T \int_{\mathbb{G}} \varphi_R(t, x)|u(t, x)|^p \;dx\;dt\\\nonumber&\quad +C	\int_0^T \int_{\mathbb{G}}  \left(|\partial_t^2 \varphi_R(t, x)|^{\frac{p}{p-1}}+|\mathcal{R}\varphi_R(t, x)|^{\frac{p}{p-1}}+|\partial_t \varphi_R(t, x)|^{\frac{p}{p-1}}\right) \varphi_R(t, x)^{-\frac{1}{p-1}} \;d x \;d t\\
	&\leq \frac{1}{2}\int_0^T \int_{\mathbb{G}} \varphi_R(t, x)|u(t, x)|^p \;dx\;dt 
	+C R^{-\frac{\nu p}{p-1}+Q+\nu},
\end{align}
where we used the fact  that $\partial_t \varphi_R \leq 0$. 	Thus  from     (\ref{eq30}), we  have 
\begin{align}\label{eq31}
	\frac{1}{2}	\int_0^T \int_{\mathbb{G}}|u(t, x)|^p \varphi_R(t, x) \;d x \;d t\leq C R^{-\frac{\nu p}{p-1}+Q+\nu} -\varepsilon \int_{\mathbb{G}}(u_0(x)+u_1(x) ) \varphi_R(0, x) \;d x.
\end{align} 		
Now we   find  estimate for $$	\int_{\mathbb G}\left(u_{0}(x)+u_{1}(x)\right) \varphi_R(0,x)\;dx$$  using the assumption (\ref{eq32}).   Prior to that, under the assumptions (\ref{eq32}), we must guarantee that the set of all initial data $\left(u_0, u_1\right)$ in $ \dot {H}^{-\gamma} \times \dot H^{-\gamma}$ is not empty. To prove this, first we claim that $\mathcal{D}_{Q, \gamma} \cap\left(L^{\frac{2Q}{Q+2\gamma}} \times L^{\frac{2Q}{Q+2\gamma}}  \right) \neq \emptyset$ for $\frac{2Q}{Q+2\gamma}>1$, where     the  set $\mathcal{D}_{Q, \gamma}$ is given by  	
\begin{align*}
	\mathcal{D}_{Q, \gamma}&:=\{\left(u_{0}, u_{1}\right)\in L^1_{\text{loc}}(\mathbb{G}) \times L^1_{\text{loc}}(\mathbb{G}): u_0(x)+u_1(x) \geq C_1 \langle x \rangle^{-Q\left(\frac{1}{2}+\frac{\gamma}{Q}\right)}(\log (e+|x|))^{-1}\}.
\end{align*} In particular, consider the functions	  
\begin{align*}
	u_{0}(x)=u_{1}(x)&=  C_1  \langle x \rangle^{-Q\left(\frac{1}{2}+\frac{\gamma}{Q}\right)}(\log (e+|x|))^{-1}.
\end{align*}
Then using the polar decomposition   (\ref{polardeco}) on   $\mathbb{G}$, we obtain
$$
\begin{aligned}
	\int_{\mathbb G}|u_{0}(x)| ^{\frac{2Q}{Q+2\gamma}}  {d} x& =  C_1^{\frac{2Q}{Q+2\gamma}}  
	\int_{\mathbb G}\langle x \rangle^{-Q\left(\frac{1}{2}+\frac{\gamma}{Q}\right)\times \frac{2Q}{Q+2\gamma}}(\log (e+|x|))^{-\frac{2Q}{Q+2\gamma}}\;dx\\ 
	& =  C_1^{\frac{2Q}{Q+2\gamma}}  \int_{\mathbb G}\langle x \rangle^{-Q }(\log (e+|x|))^{-\frac{2Q}{Q+2\gamma}}\;dx\\ 
	&  \lesssim    C_1^{\frac{2Q}{Q+2\gamma}} \int_{0}^{\infty} \langle r \rangle^{-Q} r^{Q-1}(\log ( {e}+ r))^{-\frac{2Q}{Q+2\gamma}}dr\\
	& =  C_1^{\frac{2Q}{Q+2\gamma}} \int_{0}^{\infty} \langle r \rangle^{-1}  (\log ( {e}+ r))^{-\frac{2Q}{Q+2\gamma}}dr<\infty,
\end{aligned}
$$ for $\frac{2Q}{Q+2\gamma}>1$.
This shows that  $u_{0}, u_{1} \in L^{\frac{2Q}{Q+2\gamma}} $. Since $\frac{Q+2 \gamma}{2 Q}-\frac{1}{2}=\frac{\gamma}{Q}$ with $\gamma \in\left(0, \frac{Q}{2}\right)$, according to the Hardy-Littlewood-Sobolev inequality (\ref{eq177}), we have $ L^{\frac{2Q}{Q+2\gamma}}  \subset \dot {H}^{-\gamma}$.  Consequently, in any case, we may claim that
$$
\mathcal{D}_{Q, \gamma} \cap\left(\dot {H}^{-\gamma}\times \dot {H}^{-\gamma} \right) \neq \emptyset, \quad \text{for} ~\gamma \in\left(0, \frac{Q}{2}\right).
$$	
Now from our assumption (\ref{eq32}), for $R \gg 1$,   we obtain 
\begin{align}\label{eq33}\nonumber
	&	\int_{\mathbb G }\left(u_{0}(x)+u_{1}(x)\right) \varphi_R(0,x)\;dx \\\nonumber
	& \geq C_1 \int_{|x|\leq R}    \langle x \rangle^{-Q\left(\frac{1}{2}+\frac{\gamma}{Q}\right)}(\log (e+|x|))^{-1} dx \\
	& \geq C_1R^{Q-Q(\frac{1}{2}+\frac{\gamma}{Q} )}(\log R)^{-1}=C_1 R^{\frac{Q}{2}-\gamma} (\log R)^{-1}.
\end{align}	
By denoting $$	I_R=	\int_0^T \int_{\mathbb G}|u(t, x)|^p \varphi_R(t, x) \;d x \;d t,$$  from (\ref{eq31}) and (\ref{eq33}), we have
\begin{align}\label{eq34} 
	0\leq 	\frac{I_R}{2}    \leq C R^{-\frac{\nu p}{p-1}+Q+\nu}-C_1 \varepsilon R^{\frac{Q}{2}-\gamma}(\log R)^{-1}. 
\end{align} 
By the assumption $1<p <p_{\text {Crit }}(Q, \gamma, \nu)=1+\frac{2\nu}{Q+2\gamma},$ we have $Q+\nu-\nu p'-(\frac{Q}{2}-\gamma )<0$ and consequently,  it is easy to see  that 
\begin{equation}\label{bhah}
	R^{Q+\nu-2p'-\frac{Q}{2}+\gamma} \log R < \frac{\varepsilon C_1}{C},
\end{equation} for large $R \gg 1.$ Thus  from \eqref{bhah} and \eqref{eq34}, we obtain 
$$0\leq 	\frac{I_R}{2}    \leq C R^{-\frac{\nu p}{p-1}+Q+\nu}-C_1 \varepsilon R^{\frac{Q}{2}-\gamma}(\log R)^{-1}<0,$$
which is a contradiction and this completes the proof of the blow-up part. \end{proof}

\begin{rem}
In order to ensure that the function $\varphi_R$ belongs to the space $\mathcal{C}_0^{\infty}([0, T) \times \mathbb{G})$, the scaling factor $R^\nu$ in the bump function with respect to the time variable must be dominated by the lifespan $T_{w,\varepsilon}$ of the weak solution.  To obtain an upper bound estimate for the lifespan, we consider increasing $R \uparrow T_{w,\varepsilon}^{ \frac{1}{\nu}}$ in (\ref{eq34}).  Consequently, similar to \eqref{eq34}, a contradiction  exists if  we have  $$
{T_{w,\varepsilon}^{\frac{Q+\nu}{\nu}-p'}}<C \varepsilon T_{w,\varepsilon}^{\frac{{\frac{Q}{2}-\gamma}}{\nu}}(\log T_{w,\varepsilon})^{-1},	$$		i.e., 
\begin{align}\label{upper}
	T_{w,\varepsilon} \leq C\varepsilon^{-\left(\frac{1}{p-1}-\left(\frac{Q}{2\nu}+\frac{\gamma}{\nu}\right)\right)^{-1}},
\end{align}
where  $C$ is positive constant and independent of $\varepsilon$ and $ p$.  Therefore (\ref{upper}) offers an upper bound estimate for the  lifespan for the local in time weak solutions to the Cauchy problem  (\ref{eq0010}).  
\end{rem}

\subsection{Sharp lifespan estimate} 		 From  (\ref{upper}) we have an upper bound estimate for the  lifespan for the local in time weak solutions to (\ref{eq0010}). In order to find a sharp estimate for the lifespan, our next aim is to find the lower bound estimates of lifespan.   We will employ several notations, namely,  the particular data space $\mathcal{A}^1$ and the particular evolution space $X_1(T)$, which were considered in Section \ref{sec3}.   Let   $T_{m,\varepsilon}$ be the lifespan of a mild solution $u $  to the Cauchy problem (\ref{eq0010}). Then, we have the following result regarding the lower bound for the lifespan of mild solution.

\begin{theorem}\label{lower bound}  Let ${\mathbb{G}}$ be a graded Lie group of  homogeneous dimension $Q$ and let $\mathcal{R}$ be a positive Rockland operator of homogeneous degree $\nu \geq 2.$
Let $\gamma \in (0, \tilde{\gamma} ) $ and let the  exponent $p$ satisfy $1<p<p_{\text {Crit }}(Q, \gamma, \nu)
$ such that
\begin{align}\label{eq35}
	1+\frac{2 \gamma}{Q} \leq p \left\{\begin{array}{ll}
	<\infty & \text { if } Q \leq 2, \\
	\leq   \frac{Q}{Q-2}& \text { if } Q>2.
\end{array}\right. 
\end{align}
We also   assume that  $\left(u_0, u_1\right) \in \mathcal{A}^1 $ such that  $\|(u_0, u_1)\|_{\mathcal{A}^1 }<\varepsilon$.  Then, there exists a constant $\varepsilon_0$ such that for every $\varepsilon \in\left(0, \varepsilon_0 \right]$, the lifespan $T_{m, \varepsilon}$ of  mild solutions  $u$ to the Cauchy problem (\ref{eq0010}) satisfies the following lower bound condition:
$$
T_{m,\varepsilon} \geq  C\varepsilon^{-\left(\frac{1}{p-1}-\left(\frac{Q}{2\nu}+\frac{\gamma}{\nu}\right)\right)^{-1}},
$$
where the positive constant  $C$ is  independent of $\varepsilon$, but  may  depends on $p, Q, \gamma$ as well as $\left\|\left(u_0, u_1\right)\right\|_{\mathcal{A}^1 }$.
\end{theorem} 
\begin{proof} 
Here, we observe that local-in-time mild solutions can be directly expressed and estimated using the definition of mild solutions (\ref{f2inr}).  By following a similar approach as the one used to derive estimate (\ref{eq222222}) in Section \ref{sec4}, we obtain
\begin{align}\label{eq27}\nonumber
	(1+t)^{\frac{\gamma}{\nu}}\|u(t, \cdot)\|_{L^2} \lesssim &  \left\|\left(u_0, u_1\right)\right\|_{\mathcal{A}^1 }  +  \int_0^{\frac{t}{2}}(1+\kappa)^{-\frac{p}{\nu} (  {\gamma}+\frac{Q}{2} )+  \frac{1}{\nu}\left(  {\gamma}+\frac{Q}{2}\right)  } d\kappa ~  \left\| u\right\|_{ X_1(T)}^p\\\nonumber
	& \qquad +(1+t)^{-\frac{p}{\nu} (  {\gamma}+\frac{Q}{2} )+  \frac{1}{\nu}\left(  {\gamma}+\frac{Q}{2}\right) +\frac{\gamma}{\nu} }  \int_{\frac{t}{2}}^t(1+t-\kappa)^{-\frac{\gamma}{\nu}} \mathrm{~d} \kappa\|u\|_{X_1(T)}^p\\\nonumber
	\lesssim &  \left\|\left(u_0, u_1\right)\right\|_{\mathcal{A}^1 }  +  \int_0^{\frac{t}{2}}(1+\kappa)^{-\frac{p}{\nu} (  {\gamma}+\frac{Q}{2} )+  \frac{1}{\nu}\left(  {\gamma}+\frac{Q}{2}\right)  } d\kappa ~  \left\| u\right\|_{ X_1(T)}^p\\
	& \qquad +(1+t)^{-\frac{p}{\nu} (  {\gamma}+\frac{Q}{2} )+  \frac{1}{\nu}\left(  {\gamma}+\frac{Q}{2}\right) +1 }  \|u\|_{X_1(T)}^p,
\end{align}
because of $\gamma \in (0, \min\{\nu, \frac{Q}{2}\} )$.  Note that, from (\ref{eq17}),  the  restriction  on  $p$ as 
$$
1+\frac{2\gamma}{Q} \leq p  \left\{\begin{array}{ll}
	<\infty & \text { if } Q \leq 2, \\
	\leq   \frac{Q}{(Q-2)}& \text { if } Q>2,
\end{array}\right. 
$$
is due to the 	 application of the Gagliardo-Nirenberg inequality.  
Due to the fact that  $
1<p<p_{\text {Crit }}(Q, \gamma, \nu)=1+\frac{2\nu}{Q+2\gamma}$, we immediately get $$-\frac{p}{\nu} (  {\gamma}+\frac{Q}{2} )+  \frac{1}{\nu}\left(  {\gamma}+\frac{Q}{2}\right)  > -  \frac{2\gamma+Q}{2\nu}   \left(1+\frac{2\nu}{Q+2\gamma}\right) + \frac{2\gamma+Q}{2\nu}=-1
$$
and   from (\ref{eq27}), we have 
\begin{align}\label{eq28}
	(1+t)^{\frac{\gamma}{2}}\|u(t, \cdot)\|_{L^2} \lesssim  \left\|\left(u_0, u_1\right)\right\|_{\mathcal{A}^1}+(1+t)^{-\frac{p}{\nu} (  {\gamma}+\frac{Q}{2} )+  \frac{1}{\nu}\left(  {\gamma}+\frac{Q}{2}\right) +1 }  \|u\|_{X_1(T)}^p.
\end{align}
Choosing $s=1$ in  (\ref{eq22}), we get 
\begin{align}\label{eq29}
	(1+t)^{\frac{1+\gamma}{\nu}}	\left\|u(t, \cdot)\right\|_{\dot {H}^1}
	\lesssim  \left\|\left(u_0, u_1\right)\right\|_{\mathcal{A}^1}
	+(1+t)^{-\frac{p}{\nu} (  {\gamma}+\frac{Q}{2} )+\frac{Q}{2\nu}  +\frac{\gamma}{\nu}+1} \|u\|_{X_1(T)}^p .
\end{align}
Thus from  (\ref{eq28}) and (\ref{eq29}), we can write 
\begin{align}\label{eq2999}
	\|u\|_{X_1(T)} \leq   C\varepsilon +D (1+t)^{\alpha(p, Q, \gamma, \nu)}\|u\|_{X_1(T)}^p,
\end{align}
where $\alpha(p, Q, \gamma, \nu):=-\frac{p}{\nu} (  {\gamma}+\frac{Q}{2} )+\frac{Q}{2\nu}  +\frac{\gamma}{\nu}+1>0$ and $C , D$ are two positive constants independent of $\varepsilon$ and $T$.

Now 	for a sufficiently large   constant $M>0$ (to be choosen  later),	let us   introduce  
$$
T^*:=\sup \left\{T \in\left[0, T_{m,\varepsilon}\right) \text { such that } \mathcal{X}(T):=\|u\|_{X_1(T)} \leq M \varepsilon\right\}.
$$
Using the fact that  $\mathcal{X}\left(T^*\right) \leq M \varepsilon$, from (\ref{eq2999}),	  we obtain 
\begin{align}\label{eq299}\nonumber
	\mathcal{X}\left(T^*\right)=\|u\|_{X_1(T^*)} 
	& \leq \left( C\varepsilon+D\left(1+T^*\right)^{\alpha(p, Q, \gamma, \nu)} M^p \varepsilon^{p}\right)\\ 
	& = \varepsilon \left( C+D\left(1+T^*\right)^{\alpha(p, Q, \gamma, \nu)} M^p \varepsilon^{p-1}\right)	
	<M \varepsilon,
\end{align}
for sufficiently large $M$,   provided  $2C<M$  and 
$$
2D\left(1+T^*\right)^{\alpha(p,Q, \gamma, \nu)} M^{p-1} \varepsilon^{p-1}<1 .
$$
It is crucial to observe that the function $\mathcal{X}=\mathcal{X}(T)$ is continuous for $T \in\left(0, T_{m,\varepsilon}\right)$.  	However, inequality (\ref{eq299}) provides  the existence of a time $T_0 \in \left(T^*, T_{m,\varepsilon} \right)$ such that $\mathcal{X}\left(T_0\right) \leq M \varepsilon$. This contradicts the fact that     $T^*$ is the supremum.
In other words, we need to ensure strict compliance with the following requirement:
$$
D\left(1+T^*\right)^{\alpha(p, Q, \gamma, \nu)} M^{p-1} \varepsilon^{p-1} \geq 1,
$$
that is 
$$
1+T^*   \geq D^{-\frac{1}{\alpha(p, Q, \gamma, \nu)}} M^{-\frac{p-1}{\alpha(p, Q, \gamma, \nu )}}\varepsilon^{-\frac{p-1}{\alpha(p, Q, \gamma, \nu)}}.
$$
Consequently, we can deduce  the estimated time of the blow-up  as 
\begin{align}\label{lower}
	T_{m,\varepsilon} \geq C \varepsilon^{-\frac{p-1}{\alpha(p, Q, \gamma, \nu )}}
	= C \varepsilon^{-\frac{\nu(p-1)}{(1-p) (  {\gamma}+\frac{Q}{2} )+\nu}}=C \varepsilon^{-\left(\frac{1}{p-1}-\left(\frac{Q}{2\nu}+\frac{\gamma}{\nu}\right)\right)^{-1}}, 
\end{align}
which is the required lower-bound estimate for the lifespan of a mild solution.  \end{proof}

\section*{Acknowledgement}
VK and MR are supported by the FWO Odysseus 1 grant G.0H94.18N: Analysis and Partial Differential Equations, the Methusalem program of the Ghent University Special Research Fund (BOF) (Grant number 01M01021), and by FWO Senior Research Grant G011522N.  
MR is also supported by EPSRC grant EP/V005529/1.  SSM also thanks Ghent Analysis \& PDE Center of Ghent University for the financial support of his visit to Ghent University during which this work was started. SSM is also supported by the DST-INSPIRE Faculty Fellowship DST/INSPIRE/04/2023/002038.

		\section{Data availability statement}
		The authors confirm that the data supporting the findings of this study are available within the article  and its supplementary materials.

		\section{Declarations}\vspace{0.1cm}
	 \textbf{Ethical Approval:} Not applicable.\vspace{0.1cm}

		\textbf{Competing interests:} No potential competing of interest was reported by the author.\vspace{0.1cm}


		\textbf{Availability of data and materials:} All the data uned are  within the manuscript.


\begin{thebibliography}{99}


\bibitem{Fermanian}	H. Bahouri, C. Fermanian-Kammerer and I. Gallagher, Dispersive estimates for the Schrödinger operator on step-2 stratified Lie groups, \emph{Anal. PDE} 9(3), 545–574  (2016). 

\bibitem{Bahouri} H. Bahouri and  P. G\'erard, C.-J. Xu, Espaces de Besov et estimations de Strichartz g\'en\'eralisées sur le groupe de Heisenberg, \emph{J. Anal. Math.} 82,   93–118 (2000).





\bibitem{BKM22} A. K. Bhardwaj, V. Kumar and S. S. Mondal, Estimates for the nonlinear viscoelastic damped wave equation on compact Lie groups, \emph{Proc. Roy. Soc. Edinburgh Sect. A.} (2023). \url{https://doi.org/10.1017/prm.2023.38}  



\bibitem{Reissig} 	W.  Chen and M.  Reissig,   On the critical exponent and sharp lifespan estimates for semilinear damped wave equations with data from Sobolev spaces of negative order, \emph{J. Evol. Equ.} 23, 13 (2023).  

\bibitem{Tao} M. Christ, J. Colliander, and T. Tao, A priori bounds and weak solutions for the nonlinear Schrödinger equation in Sobolev spaces of negative order, \emph{J. Funct. Anal.}  254(2),   368-395 (2008). 


\bibitem{DKM23} A. Dasgupta, V. Kumar and S. S. Mondal, Nonlinear fractional damped wave equations on compact Lie groups, \emph{Asymptot. Anal.} 134(3-4), 485-511, (2023).  

\bibitem{DKMR}  A. Dasgupta, V. Kumar, S. S. Mondal  and M. Ruzhansky, Semilinear damped wave equations on the Heisenberg group with initial data from Sobolev spaces of negative order, to appear in {\it  J. Evol. Equ.} (2024).
\url{https://doi.org/10.48550/arXiv.2401.06565}
 
\bibitem{Wang} L. Du and H. Wang Quasigeostrophic equation with random initial data in negative order Sobolev space, \emph{Z. Angew. Math. Phys.}  70, 95 (2019).
 



\bibitem{Fischer}	V. Fischer and  M. Ruzhansky, \emph{Quantization on Nilpotent Lie Groups}, Progr. Math., vol.314, Birkh\"auser/Springer, (2016).

\bibitem{RF17}  V. Fischer and M. Ruzhansky,  Sobolev spaces on graded Lie groups,  \emph{Ann. Inst. Fourier (Grenoble)} 67(4), 1671–1723 (2017).

 \bibitem{Folland} 	G.B. Folland and E.M. Stein, Hardy spaces on homogeneous groups, Princeton University	Press, Princeton, New Jersey (1982).

\bibitem{FMV} 	G. Furioli, C. Melzi and A. Veneruso, Strichartz inequalities for the wave equation with the full Laplacian on the Heisenberg group, \emph{Canad. J. Math.} 59(6), 1301–1322 (2007).
 

\bibitem{garetto} 	C.	 Garetto and M.  Ruzhansky, Wave equation for sums of squares on compact Lie groups, \emph{J. Differential Equations} 258(12),  4324-4347 (2015).		

\bibitem{Vla}	 V. Georgiev and A. Palmieri, Critical exponent of Fujita-type for the semilinear damped wave equation on the Heisenberg group with power nonlinearity, \emph{J. Differential Equations} 269(1), 420-448 (2020).



\bibitem{GW} Y.  Guo and Y.  Wang, Decay of dissipative equations and negative Sobolev spaces, \emph{Comm. Partial Differential Equations} 37(12), 2165-2208 (2012).

\bibitem{HN} B. Helffer and J, Nourrigat, Caracterisation des operateurs hypoelliptiques homogenes invariants a gauche sur un groupe de Lie nilpotent gradue. {\it Comm. Partial Differential Equations} 4(8), 899–958, (1979).

\bibitem{spectrum} A. Hulanicki, J. W. Jenkins, and J. Ludwig, Minimum eigenvalues for positive, Rockland operators, \emph{Proc. Amer. Math. Soc.} 94,  718–720 (1985).

\bibitem{Ikeda2019} M. Ikeda, T. Inui, M. Okamoto and  Y. Wakasugi, $L^p-L^q$ estimates for the damped wave equation and the critical exponent for the nonlinear problem with slowly decaying data, \emph{Commun. Pure Appl. Anal.} 18(4), 1967-2008 (2019).











\bibitem{Ikeda2002} R. Ikehata and M. Ohta, Critical exponents for semilinear dissipative wave equations in $\mathbb{R}^N$, \emph{J. Math. Anal. Appl.} 269(1), 87-97 (2002).	 			

\bibitem{IKeta and Tanizawa}  R. Ikehata and K. Tanizawa, Global existence of solutions for semilinear damped wave equations in $\mathbb{R}^n$ with noncompactly supported initial data, \emph{Nonlinear Anal.} 61(7), 1189-1208  (2005).


 

\bibitem{AKR22} A. Kassymov, V. Kumar and M. Ruzhansky,  Functional inequalities on symmetric spaces of noncompact type and applications, to appear in {\it J. Geom. Anal.} (2024). \url{https://doi.org/10.48550/arXiv.2212.02641} 

\bibitem{attract1}  A. E. Kogoj and S. Sonner, Attractors met X-elliptic operators, \emph{J. Math. Anal. Appl.} 420(1),  407–434 (2014).




\bibitem{attract2}  D. Li, Ch. Sun, and Q. Chang, Global attractor for degenerate damped hyperbolic equations, \emph{J. Math. Anal. Appl.} 453(1), 1–19 (2017).


			\bibitem{32}     Y. Liu, Y. Li, and J. Shi, Estimates for the linear viscoelastic damped wave equation on the Heisenberg group,  \emph{J. Differential Equations} 285, 663–685 (2021).
   
\bibitem{Manli} H. Liu and M. Song, Strichartz Inequalities for the Wave Equation with the Full Laplacian on $H$-Type Groups,  \emph{Abstr. Appl. Anal.} 2014, Art. ID 219375 (2014).
 

\bibitem{Matsumura}  A. Matsumura, On the asymptotic behavior of solutions of semi-linear wave equations, \emph{Publ. Res. Inst. Math. Sci.} 12(1), 169-189 (1976/77).

\bibitem{Miller} K. G. Miller, Parametrices for hypoelliptic operators on step two nilpotent Lie groups, {\it Comm. Partial Differential
	Equations}, 5(11), 1153–1184, (1980).

 \bibitem{MS99} D. Müller and E. M. Stein,  $L^p$-estimates for the wave equation on the Heisenberg group, {\it Revista Matemática Iberoamericana}, 15(2), 297-332,  (1999).

\bibitem{MS15} D. Müller and A. Seeger, Sharp $L^p$ bounds for the wave equation on groups of Heisenberg type, {\it Anal. PDE},  8(5), 1051-1100, (2015).

\bibitem{24} A. I. Nachman, The wave equation on the Heisenberg group, \emph{Comm. Partial Differential Equations} 7(6),   675-714 (1982).


\bibitem{Nakao93}  M. Nakao and  K. Ono, Existence of global solutions to the Cauchy problem for the semilinear dissipative wave equations, \emph{Math. Z.} 214(2), 325-342 (1993).

 

\bibitem{Ohhh}  T. Oh and Y. Wang,
Global well-posedness of the periodic cubic fourth order NLS in negative Sobolev spaces, \emph{Forum Math. Sigma} 6, Article e5 (2018). 

 \bibitem{Palmieri 2020}	A. Palmieri, Decay estimates for the linear damped wave equation on the
		Heisenberg group, \emph{J. Funct. Anal.}, 279(9), 108721 (2020).

\bibitem{27}   A. Palmieri,  On the blow-up of solutions to semilinear damped wave equations with power nonlinearity in compact Lie groups, \emph{J. Differential Equations } 281, 85-104 (2021). 


\bibitem{28}  A.  Palmieri,  A global existence result for a semilinear wave equation with lower order terms on compact Lie groups,  \emph{J. Fourier Anal. Appl. }28,  Article number: 21 (2022).



\bibitem{Rock} C. Rockland, Hypoellipticity on the Heisenberg group-representation-theoretic criteria. \emph{Trans. Amer. Math. Soc.},
240, 1–52, (1978).

\bibitem{gra1}   M. Ruzhansky and C. Taranto, Time-dependent wave equations on graded groups, \emph{Acta Appl. Math.} 171, Article number: 21 (2021).			


\bibitem{30} M. Ruzhansky and N. Tokmagambetov, Nonlinear damped wave equations for the sub-Laplacian on the Heisenberg group and for Rockland operators on graded Lie groups, \emph{J. Differential Equations} 265(10), 5212-5236 (2018).

\bibitem{RY20} M. Ruzhansky and N. Yessirkegenov, Very weak solutions to hypoelliptic wave equations, \emph{ J. Differential Equations}, 268, 2063-2088 (2020).





\bibitem{Yao}  A. Soffer, Y. Wu, and X. Yao, Global rough solution for $L^2$-critical semilinear heat equation in the negative Sobolev space, (2019). arxiv preprint. \url{https://doi.org/10.48550/arXiv.1903.08316}.

\bibitem{manli1} M. Song and  J. Yang,  Decay estimates for a class of wave equations on the Heisenberg group, \emph{Ann. Mat. Pura Appl.} 202(6),    2665–2685  (2023).


\bibitem{gra3}   C. Taranto, \emph{Wave equations on graded groups and hypoelliptic Gevrey spaces}, Imperial College London Ph.D. thesis, (2018). 

\bibitem{terRob97} A. F. M. ter Elst and Derek W. Robinson,  Spectral estimates for positive Rockland operators. Algebraic groups and Lie groups, 195–213, Austral. Math. Soc. Lect. Ser., 9, Cambridge Univ. Press, Cambridge, 1997



\bibitem{Todorova}  G. Todorova and  B. Yordanov, Critical exponent for a nonlinear wave equation with damping, \emph{J. Differential Equations} 174(2), 464-489 (2001). 
 
 
\bibitem{Umakoshi}  H. Umakoshi, A semilinear heat equation with initial data in negative Sobolev spaces. \emph{Discrete Contin. Dyn. Syst. Ser. A} 14(2), 745-767 (2021).

\bibitem{Yang} Z. Yang,  Fujita exponent and nonexistence result for the Rockland heat equation, \emph{Appl. Math. Lett.} 121, 107386 (2021).

\bibitem{Zhang} 	 Q. S. Zhang, A blow-up result for a nonlinear wave equation with damping: the critical case, \emph{C. R. Acad. Sci. Paris Sér. I Math.} 333(2), 109-114  (2001). 

\end{thebibliography}
\end{document}